\documentclass[11pt,a4paper]{article}
\usepackage{amsmath,amssymb,amsthm,mathrsfs}

\renewcommand{\mathcal}{\mathscr}

\renewcommand{\P}{\mathrm{P}}
\newcommand {\E}{{\mathrm E}}
\newcommand{\1}{{\bf 1}}

\newcommand{\R}{\mathbb{R}}

\newtheorem{stat}{Statement}[section]
\newtheorem{prop}[stat]{Proposition}

\newtheorem{thm}[stat]{Theorem}
\newtheorem{lem}[stat]{Lemma}

\theoremstyle{definition}
\newtheorem{remark}[stat]{Remark}
\numberwithin{equation}{section}

\begin{document}
\title{ \bf On the density of systems of non-linear spatially homogeneous SPDEs}

\author{Eulalia Nualart$^{1}$}

\date{}

\maketitle

\footnotetext[1]{Institut Galil\'ee, Universit\'e
    Paris 13, 93430 Villetaneuse, France.
    \texttt{eulalia@nualart.es}, http://nualart.es}

\maketitle
\begin{abstract}
In this paper, we consider a system of $k$ second order non-linear
stochastic partial differential equations with spatial dimension $d \geq 1$,
driven by a $q$-dimensional Gaussian noise, which is white in time
and with some spatially homogeneous covariance.
The case of a single equation and a one-dimensional noise,
has largely been studied in the literature.
The first aim of this paper is to give a survey of some of the existing results. We will start with the
existence, uniqueness and H\"older's continuity of the solution. For this, the extension of Walsh's stochastic integral
to cover some measure-valued integrands will be recalled.
We will then recall the results concerning the
existence and smoothness of the density, as well as its strict positivity, which are obtained using techniques of Malliavin calculus.
The second aim of this paper is to show how these results extend to our system of SPDEs.
In particular, we give sufficient conditions in order to have existence and smoothness of the density on the set where the columns of the diffusion matrix
span $\R^k$. We then prove that the density is strictly positive in a point if the connected component of the set where the columns of the diffusion matrix
span $\R^k$ which contains this point has a non void intersection with the support of the law of the solution.
We will finally check how all these results apply to the case
of the stochastic heat equation in any space dimension and the
stochastic wave equation in dimension $d\in \{1,2,3\}$.
\end{abstract}

\vskip 1,5cm {\it \noindent AMS 2000 subject classifications:} 60H15, 60H07. \vskip 10pt

\noindent {\it Key words and phrases}. Spatially homogeneous Gaussian noise, Malliavin calculus, non-linear stochastic partial differential equations, strict positivity of the density.  \vskip 4cm \pagebreak

\section{Introduction}
Consider the system of stochastic partial differential equations:
\begin{equation} \label{equa1}
Lu_i (t,x)= \sum_{j=1}^{q} \sigma_{ij}(u(t,x)) \dot{W}^{j}(t,x) +
b_i(u(t,x)), \; \; t \geq 0, \; x \in \R^d,
\end{equation}
$i=1,...,k$, with vanishing initial conditions. Here $L$ is a second
order differential operator, and $\sigma_{i j},b_i:\R^{k} \mapsto \R$
are globally Lipschitz functions, which are the entries of a $k\times q$ matrix $\sigma$ and a $k$-dimensional vector $b$.
We denote by $\sigma_1,...,\sigma_q$ the columns of the matrix $\sigma$.
The driving perturbation
$\dot{W}(t,x)=(\dot{W}^1(t,x),...,\dot{W}^{q}(t,x))$ is a
$q$-dimensional Gaussian noise which is white in time and with a
spatially homogeneous covariance $f$, that is,
\begin{equation*}
\E \,  [\dot{W}^i(t,x) \dot{W}^j(s,y)]= \delta(t-s) f(x-y) \delta_{ij},
\end{equation*}
where $\delta(\cdot)$ denotes the Dirac delta function, $\delta_{ij}$ the Kronecker symbol, and
$f$ is a positive continuous function on $\R^d \setminus \{0\}$.

The basic examples we are interested in are the stochastic wave and heat equations with vanishing initial conditions,
that is,
$L=\frac{\partial^2}{\partial t^2}-\Delta$ and $L=\frac{\partial}{\partial t}-\Delta$, where $\Delta$ denotes the
Laplacian
operator in $\R^d$, and the spatial covariance $f$ to be the Riesz kernel, that is, $f(x)=\Vert x \Vert^{-\beta}$, $0<\beta <d$.

Let $(\mathcal{F}_t)_{t \geq 0}$ denote the filtration generated by
$W$, and let
$T>0$ be fixed. By definition, the solution to the formal equation
(\ref{equa1}) is an adapted stochastic process $\{
u(t,x)=(u_1(t,x),...,u_{k}(t,x)), (t,x) \in [0,T] \times \R^d\}$
such that
\begin{equation} \label{equa2}
\begin{split}
u_i(t,x)&= \sum_{j=1}^{q} \int_0^t \int_{\R^d} \Gamma(t-s, x-y)  \sigma_{ij} (u(s,y)) W^j(ds, dy) \\
&\qquad +\int_0^t \int_{\R^d} b_i(u(t-s,x-y)) \Gamma(s, dy) ds,
\end{split}
\end{equation}
where $\Gamma$ denotes the fundamental solution of the deterministic equation $Lu=0$.

Recall that when $\Gamma(t,x)$ is a real valued function, the
stochastic integral appearing in (\ref{equa2}) is the classical
Walsh stochastic integral (see \cite{Walsh:86}). However, when
$\Gamma$ is measure, Dalang \cite{Dalang:99} extended Walsh's
stochastic integral using techniques of Fourier analysis, and
covered, for instance, the case of the wave equation in dimension
three. D.Nualart and Quer-Sardanyons \cite{Nualart:07} extend
Walsh's stochastic integral using techniques of stochastic
integration with respect to a cylindrical Wiener process (see
\cite{Daprato:92}), in order to cover some classes of measure-valued
integrands. 

In this paper, we are interested in studying the existence, smoothness and strict positivity of the density
of the solution to the system of SPDEs (\ref{equa1}), on the set where $\{\sigma_1,...,\sigma_q\}$ span $\R^k$. The case of a single equation 
has largely been studied in the literature, so our aim is to give a survey of the known results and explain how they
extend to the case of a system of equations.

One of the motivations of the results of this article, is to develop
in a further work potential theory for solutions of systems of the
type (\ref{equa1}) (see \cite{Dalang:09b}). For the moment
potential theory for systems of non-linear SPDEs has been studied by
Dalang and E.Nualart in \cite{Dalang:04} for the wave equation, and
by Dalang, Khoshnevisan and E.Nualart in \cite{Dalang:08} and
\cite{Dalang:09}, for the heat equation, all driven by space-time
white noise. That is, taking in equation (\ref{equa1}), $d=1$, $f$
the Dirac delta function and $L$ the fundamental solution of the
deterministic wave or heat equation. In all these works, the
existence, smoothness, and strict positivity
of the density of the solution to the system of SPDEs is required.

The paper is organized as follows. Section $2$ deals with the existence and uniqueness of the solution
to (\ref{equa2}). For this, we first need to define rigorously the noise $W$, and explain in which sense the stochastic integral
in (\ref{equa2}) is understood, recalling the mentioned results in \cite{Dalang:99} and \cite{Nualart:07}. 
Studying the density of an SPDE is sometimes related to the H\"older continuity properties of the paths of its solution.
Thus, we will recall some results concerning the H\"older continuity of the solution 
to the stochastic heat and wave equations. This will be done in Section 3.
Section 4 will be devoted to the study of the existence and smoothness of the density of the solution to (\ref{equa1}). For this, some elements of Malliavin calculus
need to be introduced. Finally, the aim of Section 5 is to show the strict positivity of the density. 
This last result turns out to be new in the literature. Its proof extends the method used by Bally and Pardoux in  
\cite{Bally:98}
for the case of the stochastic heat equation driven by a space-time white noise, in our spatially homogeneous situation.

\section{The stochastic integral}

The aim of this section is to recall the results concerning the existence and uniqueness of the solution
to our class of SPDEs (\ref{equa1}). For this, we will first recall in our $q$-dimensional situation the extension of Walsh's stochastic integral given in \cite{Nualart:07}, that uses Da Prato and Zabczyk's stochastic integration theory with respect to cylindrical Wiener processes. In the recent paper \cite{Dalang:10}, Dalang and Quer-Sardanyons extensively explain how this extension is related to the pioneer work by Dalang in \cite{Dalang:99}.
  
Fix a time interval $[0,T]$. The Gaussian random perturbation
$W=(W^1,...,W^{q})$ is a $q$-dimensional zero mean Gaussian family
of random variables $W=\{W^j(\varphi), 1\leq j \leq q, \varphi \in
\mathcal{C}^{\infty}_0([0,T] \times \R^{d}; \R)\}$, defined on a
complete probability $(\Omega, \mathcal{F}, \P)$ with covariance
\begin{equation*}
\E [W^i(\varphi) W^j(\psi) ] = \delta_{ij} \int_0^T
\int_{\R^d}\int_{\R^d} \varphi(t,x) f(x-y) \psi(t,y) \, dx dy dt,
\end{equation*}
where $f: \R^d \rightarrow \R_+$ is a continuous function on $\R^d\setminus \{0\}$ which is integrable in a neighborhood of $0$.

Observe that in order that there exists a Gaussian process with this covariance functional, it is necessary and sufficient
that the covariance functional is non-negative definite. This implies that $f$ is symmetric, and is equivalent to the fact that it is the Fourier transform of a non-negative tempered measure $\mu$ on $\R^d$.
That is,
$$
f(x)= \mathcal{F} \mu (x):= \int_{\R^d} e^{-2 \pi i x \cdot \xi} \mu(d\xi),
$$
and for some integer $m \geq 1$,
$$
\int_{\R^d} \frac{\mu(d\xi)}{(1+\Vert \xi \Vert^2)^m} < +\infty.
$$

Elementary properties of the convolution and Fourier transform show that
this covariance can also be written in terms of the measure $\mu$ as
\begin{equation*}
\E [W^i(\varphi) W^j(\psi) ] =  \delta_{ij} \int_0^T  \int_{\R^d}
\mathcal{F} \varphi(t)(\xi) \overline{\mathcal{F} \psi(t) (\xi)}
\mu(d\xi) dt,
\end{equation*}
where $\overline{\mathcal{F} \psi}$ denotes the complex conjugate of $\mathcal{F} \psi$.

Let $\mathcal{H}^{q}$ denote the completion of the Schwartz space
$\mathcal{S}(\R^{d}; \R^{q})$ of $\mathcal{C}^\infty(\R^{d};
\R^{q})$ functions with rapid decrease,
 endowed with the inner product
\begin{equation*}
\langle \varphi, \psi \rangle_{\mathcal{H}^{q}}=\sum_{\ell=1}^{q}
\int_{\R^d}   \int_{\R^d}  \varphi_{\ell}(x) f(x-y) \psi_{\ell}(y)
\, dx dy =\sum_{\ell=1}^{q} \int_{\R^d} \mathcal{F}
\varphi_{\ell}(\xi) \overline{\mathcal{F} \psi_{\ell} (\xi)} \mu(d
\xi),
\end{equation*}
$\phi, \psi \in \mathcal{S}(\R^{d}; \R^{q})$. Notice that
$\mathcal{H}^{q}$ may contain distributions. Set
$\mathcal{H}^q_T=L^2([0,T]; \mathcal{H}^{q})$. In particular,
$\mathcal{H}^q_T$ is the completion of the space
$\mathcal{C}^{\infty}_0([0,T] \times \R^{d}; \R^{q})$ with respect
to the scalar product
\begin{equation*}
\langle \varphi, \psi \rangle_{\mathcal{H}_T^q}=\sum_{\ell=1}^{q}
\int_0^T \int_{\R^d}   \mathcal{F} \varphi_{\ell}(t)(\xi)
\overline{\mathcal{F} \psi_{\ell}(t) (\xi)} \mu(d \xi) dt.
\end{equation*}
Then the Gaussian family $W$ can be extended to $\mathcal{H}^q_T$ and
we will use the same notation $\{W(g), g  \in \mathcal{H}^q_T\}$,
where $W(g)=\sum_{i=1}^{q} W^i(g_i)$.

Now, set $W_t(h)=\sum_{i=1}^{q} W^i(1_{[0,t]}h_i)$,
for any $t \geq 0$, $h \in \mathcal{H}^{q}$. Then $\{W_t, t \in
[0,T]\}$ is a cylindrical Wiener process in the Hilbert space
$\mathcal{H}^{q}$ (cf. \cite[Section 4.3.1]{Daprato:92}). That is,
for any $h \in \mathcal{H}^{q}$, $\{W_t(h), t \in [0,T]\}$ is a
Brownian motion with variance $t\Vert h \Vert^2_{\mathcal{H}^{q}}$,
and
$$
\E [W_s(h) W_t(g) ]=(s \wedge t)  \langle h, g
\rangle_{\mathcal{H}^{q}}.
$$
Let $(\mathcal{F}_t)_{t \geq 0}$ denote the $\sigma$-field generated
by the random variables $\{W_s(h), h \in \mathcal{H}^{q}, 0 \leq s
\leq t\}$ and the $\P$-null sets. We define the predictable
$\sigma$-field as the  $\sigma$-field in $\Omega \times [0,T]$
generated by the sets $\{ (s,t] \times A, 0\leq s<t \leq T, A \in
\mathcal{F}_s \}$. Then following \cite[Chapter 4]{Daprato:92}, we
can define the stochastic integral of any predictable process $g \in
L^2(\Omega \times [0,T]; \mathcal{H}^{q})$ with respect to the
cylindrical Wiener process $W$ as
$$
\int_0^T \int_{\R^d} g \cdot dW:=\sum_{j=1}^{\infty} \int_0^T \langle g_t, e_j \rangle_{\mathcal{H}^q} dW_t(e_j),
$$
where $(e_j)$ is an orthonormal basis of $\mathcal{H}^q$.
Moreover, the following isometry property holds:
$$
\E \biggl[ \big\vert \int_0^T \int_{\R^d} g \cdot dW \big\vert^2\biggr]= \E \biggl[ \int_0^T \Vert
g_t \Vert^2_{\mathcal{H}^{q}} dt \biggr].
$$

We next introduce the following condition on the fundamental solution of $Lu=0$, $\Gamma$.
\begin{itemize}
\item[{\bf (H1)}] For all $t>0$, $\Gamma(t)$ is a nonnegative distribution with rapid decrease, such that
\begin{equation}  \label{hyp1}
\int_0^T \int_{\R^d} \vert \mathcal{F}\Gamma(t) (\xi) \vert^2 \mu(d
\xi) dt<+\infty.
\end{equation}
Moreover, $\Gamma$ is a nonnegative measure of the form $\Gamma(t, dx) dt$ such that
\begin{equation}  \label{hyp0}
\sup_{t \in [0,T]} \int_{\R^d} \Gamma(t, dx)  \leq C_T < +\infty.
\end{equation}
\end{itemize}

Then \cite[Lemma 3.2 and Proposition 3.3.]{Nualart:07} show that under condition {\bf (H1)}, $\Gamma$ belongs to $\mathcal{H}_T$, and
one can define the stochastic integral of a predictable process $G=G(t,dx)=Z(t,x)
\Gamma(t, dx) \in L^2(\Omega \times [0,T]; \mathcal{H}^{q})$ with
respect to $W$, provided that $Z=\{Z(t,x), (t,x) \in [0,T] \times
\R^d\}$ is a predictable process such that
\begin{equation*}
\sup_{(t,x) \in [0,T] \times \R^d} \E[\Vert Z(t,x) \Vert^2] < +\infty.
\end{equation*}
In this case, we denote the stochastic integral as
$$
\int_0^T \int_{\R^d} G(s,y) \cdot W(ds, dy) =\int_0^T \int_{\R^d} \Gamma(s,y) Z(s,y) \cdot W(ds, dy).
$$

We next state the existence and uniqueness result of the solution to
equation (\ref{equa2}). This is an extension to system of SPDEs of \cite[Theorem 4.1]{Nualart:07}
(see also \cite[Theorem 13]{Dalang:99}) and can be proved in the same way.
\begin{thm} 
Under condition {\bf (H1)}, there exists a unique adapted process
$\{ u(t,x)=(u_1(t,x),...,u_{k}(t,x)), (t,x) \in [0,T] \times \R^d\}$
solution of equation \textnormal{(\ref{equa2})}, which is continuous
in $L^2$ and satisfies that for all $T >0$ and $p \geq 1$,
\begin{equation} \label{supu}
\sup_{(t,x) \in [0,T] \times \R^d} \E[\Vert u(t,x) \Vert^p] < +\infty.
\end{equation}
\end{thm}

The basic examples we are interested in are the stochastic heat and wave equations. 
More precisely, it is well-known (see for e.g. \cite{Dalang:99}) that if $L$ is the heat operator in $\R^d$, $d \geq 1$,
that is, $L=\frac{\partial}{\partial t}-\Delta$, where $\Delta$ denotes the
Laplacian
operator in $\R^d$, or if $L$ is the wave operator in $\R^d$, $d \in \{1,2,3\}$,
that is, $L=\frac{\partial^2}{\partial t^2}-\Delta$, condition {\bf (H1)} is satisfied if and only if
\begin{equation} \label{hypot}
\int_{\R^d} \frac{\mu(d\xi)}{(1+\Vert \xi \Vert^2)} < +\infty.
\end{equation}

For instance, one can take $f$ to be a Riesz kernel, that is, $f(x)=\Vert x \Vert^{-\beta}$, $0 < \beta < d$.
Then, $\mu(d\xi)=c_{d,\beta} \Vert \xi \Vert^{\beta-d} d\xi$, and (\ref{hypot}) holds if and only if $0<\beta< (2 \wedge d)$.

\section{H\"older continuity of the solution}

Studying existence, smoothness and strict positivity of the density of the solution to the system of equations (\ref{equa1}), will require
to know the behaviour of the moments of the increments of the solution, which in particular implies the H\"older continuity
of the paths. Let us introduce the following conditions.

\begin{itemize}
\item[{\bf (H2)}] There exists $\gamma_1>0$ such that for all $t \in [0,T]$, $h \in [0, T-t]$, $x \in \R^d$, and
$p>1$,
\begin{equation*}
 \E \, [\Vert u(t+h,x)-u(t, x)\Vert^p ] \leq c_{p,T} h^{\gamma_1 p},
\end{equation*}
for some constant $c_{p,T}>0$. In particular, for all $\delta>0$ the trajectories of
$u$ are $(\gamma_1-\delta)$-H\"older continuous in time.

\item[{\bf (H3)}] There exists $\gamma_2>0$ such that for all $t \in [0,T]$, $x,y \in \R^d$, and
$p>1$,
\begin{equation*}
 \E \, [\Vert u(t,x)-u(t, y)\Vert^p ] \leq c_{p,T} \Vert x-y \Vert^{\gamma_2 p},
\end{equation*}
for some constant $c_{p,T}>0$. In particular, for all $\delta>0$ the trajectories of
$u$ are  $(\gamma_2-\delta)$-H\"older continuous in space.
\end{itemize}

Using Kolmogorov's continuity theorem, Sanz-Sol\'e and Sarr\`a prove in 
\cite{Sanz:02} that if there exists $\epsilon \in (0,1)$ such that
\begin{equation} \label{hypot2}
\int_{\R^d} \frac{\mu(d\xi)}{(1+\Vert \xi \Vert^2)^{\epsilon}} < +\infty,
\end{equation}
then the stochastic heat equation satisfies {\bf (H2)} for all $\gamma_1 \in (0, \frac{1-\epsilon}{2})$,
and {\bf (H3)} for all $\gamma_2 \in (0, 1-\epsilon)$.
In particular, if $f$ is the Riesz kernel, that is $\mu(d\xi)=c_{d,\beta} \Vert \xi \Vert^{\beta-d} d\xi$, $0<\beta<(2\wedge d)$, then {\bf (H2)} holds for all
$\gamma_1 \in (0, \frac{2-\beta}{4})$, and {\bf (H3)}
holds for all $\gamma_2 \in (0, \frac{2-\beta}{2})$. Observe that in this case (\ref{hypot2}) holds for all $\epsilon>\frac{\beta}{2}$.

For the stochastic wave equation in dimensions $d\in\{1,2\}$, one obtains that under condition (\ref{hypot2}), {\bf (H2)} and {\bf (H3)} are satisfied for all $\gamma_1, \gamma_2 \in (0, 1-\epsilon)$, that is, $\gamma_1, \gamma_2 \in (0, \frac{2-\beta}{2})$
for the Riesz kernel case. 
The two-dimensional case was studied by Millet and Sanz-Sol\'e in 
\cite[Proposition 1.4]{Millet:99}. See \cite{Walsh:86} for the one dimensional case.

A different approach based on Sobolev embedding theorem is needed to handle the stochastic wave equation in dimension three.
This was done by Dalang and Sanz-Sol\'e in \cite{Dalang:07} for the case where the spatially homogeneous covariance is defined as a product of a Riesz kernel and a smooth function. In particular, when it is a Riesz kernel, they show that {\bf (H2)} and {\bf (H3)} are satisfied with $\gamma_1, \gamma_2 \in (0, \frac{2-\beta}{2})$.

\section{Existence and smoothness of the density}

We are now interested in proving existence and smoothness of the density of the solution to the system (\ref{equa1})
on the set where $\{ \sigma_1,...,\sigma_q\}$ span $\R^k$.
As is well known, the Malliavin calculus provides a powerful tool in order to show this kind of results for solutions to SDEs and SPDEs. Let us first recall some existing results for single equations of the type (\ref{equa1}).
For the existence and regularity of the density of the wave equation with $d=1$ one refers to the work by Carmona and D.Nualart in \cite{Carmona:88}.
The wave equation in spatial dimension $2$ was studied by Millet and Sanz-Sol\'e in \cite{Millet:99}. The three dimensional case was then handled by Quer-Sardanyons and Sanz-Sol\'e, see \cite{Quer:04} and \cite{Sanz:04}.
The case of the stochastic heat equation in any space dimension can be found in \cite{Marquez:01}.
All these results where unified and generalized by D.Nualart and Quer-Sardanyons in \cite{Nualart:07}. They assumed the following condition.
\begin{itemize}
\item[{\bf (H4)}] There exists $\eta>0$ such that for all $\tau \in [0,1]$,
\begin{equation*}
c \tau^{\eta} \leq \int_0^{\tau} \int_{\R^d} \vert
\mathcal{F}\Gamma(t) (\xi) \vert^2 \mu(d \xi) dt,
\end{equation*}
for some constant $c>0$.
\end{itemize}
Then they show that if $\sigma, b$ are $\mathcal{C}^{\infty}$ functions with bounded derivatives of order greater than or equal to one, and $\vert \sigma \vert \geq c >0$, under conditions {\bf (H1)} and {\bf (H4}), for all $(t,x) \in ]0,T] \times \R^d$, the law of $u(t,x)$ has a $\mathcal{C}^{\infty}$ density.

Observe that condition {\bf (H4)} is satisfied for the stochastic heat equation for any $\eta
\geq 1$ (see \cite[Remark 6.3]{Nualart:07}), and with
$\eta=3$ for the stochastic wave equation in dimensions $1,2,3$ (see \cite[(A.3)]{Quer:04}).

The proof of this result uses tools of Malliavin calculus. One needs to check first that the random variable $u(t,x)$ is smooth in the Malliavin sense. This can be easily generalized to systems of equations (see Proposition \ref{derivat} below).
Secondly, one needs to show that the Malliavin matrix has inverse moments of all orders. For this, one uses condition
{\bf (H4)} and the non degeneracy condition on $\sigma$. The extension of this result to systems of equations is more delicate
as ones deals with a matrix $\sigma$ instead of a function, 
and because of that reason we need to consider the following additional assumptions on $\Gamma$.
\begin{itemize}
\item[{\bf (H5)}] Conditions {\bf (Hi)}, $i=1,2,3$ are satisfied. 
Moreover, the nonnegative measure $\Psi$ defined as $\Vert x \Vert^{\gamma_2} \Gamma(t,dx)$
satisfies that
\begin{equation*}
\int_0^T \int_{\R^d} \vert \mathcal{F}\Psi(t) (\xi) \vert^2 \mu(d
\xi) dt<+\infty.
\end{equation*}
Moreover, there exist $\alpha_1,\alpha_2>0$, $\gamma_2<\alpha_1$,
$\gamma_1<\alpha_2$, such that for all $\tau \in [0,1]$,
\begin{equation*}
\int_{0}^{\tau} \langle \Psi(r,\ast), \Gamma(r,\ast) \rangle_{\mathcal{H}}  dr\leq c \tau^{\alpha_1},
\end{equation*}
for some constant $c>0$, and
\begin{equation*}
\int_0^{\tau} r^{\gamma_1} \int_{\R^d} \vert \mathcal{F}\Gamma(r)
(\xi) \vert^2 \mu(d \xi) dr \leq c \tau^{\alpha_2},
\end{equation*}
for some constant $c>0$. 

\item[{\bf (H6)}] Conditions {\bf (H4)} and {\bf (H5)} are satisfied, and $\alpha:=\alpha_1 \wedge \alpha_2 >\eta$.
\end{itemize}

We will then show the following result.
\begin{thm} \label{maint0}
Assume hypothesis {\bf (H6)} and that $\sigma, b$ are $\mathcal{C}^{\infty}$ functions with bounded partial derivatives of order greater than or equal to one.
Then for all $(t,x) \in
]0,T] \times \R^d$, the law of the random vector $u(t,x)$ admits a
$\mathcal{C}^{\infty}$ density on the open subset of
$\R^{k}$ $\Sigma:=\{y \in \R^{k}: \sigma_1(y),...,\sigma_q(y) \text{ span } \R^k\}$. That is, there exists a function
$p_{t,x} \in \mathcal{C}^{\infty}(\Sigma; \R)$ such that for every
bounded and continuous function $f:\R^{k} \mapsto \R$ with support
contained in $\Sigma$,
\begin{equation*}
\E[ f(u(t,x))] =\int_{\R^{k}} f(y) \, p_{t,x}(y) dy.
\end{equation*}
\end{thm}

We next prove an intermediary result that will be needed for the
proof of Theorem \ref{maint0}.
\begin{lem} \label{mes}
Assume hypothesis {\bf
(H5)}. Then for all $t \in [0,T]$, $x \in \R^d$,
$\epsilon \in (0,1]$, and $p>1$,
\begin{equation*}
\E \biggl[\bigg\vert \int_{0}^{\epsilon} \bigg\langle (\sigma_{ij}(u(t-r,\ast))-\sigma_{ij}(u(t,x)))
\Gamma(r, x-\ast), \Gamma(r,x-\ast) \bigg\rangle_{\mathcal{H}}  dr\bigg\vert^p \biggr] \leq c_{p,T} \epsilon^{\alpha p},
\end{equation*}
for some constant $c_{p,T}>0$, where $\alpha:=\alpha_1 \wedge \alpha_2 $ and $\alpha_1, \alpha_2$ are the parameters in
hypothesis {\bf (H5)}.
\end{lem}

\begin{proof}
Set $G(r,dy):=
(\sigma_{ij}(u(t-r,y))-\sigma_{ij}(u(t,x)))\Gamma(r,x-dy)$. Then by \cite[Proposition 3.3]{Nualart:07}, $G(r,dy) \in L^2(\Omega \times [0,T]; \mathcal{H})$, and following the proof of this Proposition, 
we can write
\begin{equation*}
\langle G(r,x-\ast), \Gamma(r,x-\ast) \rangle_{\mathcal{H}}  =
\int_{\R^d} \int_{\R^d}  G(r,x-dy) f(y-z)  \Gamma(r,x-dz).
\end{equation*}
Using Minkowski's inequality with respect to the finite measure
$\Gamma(r,x-dy)\Gamma(r,x-dz)f(y-z)dr,$ together with the
Lipschitz property of $\sigma$, we get that for all $p>1$,
\begin{equation*}
\begin{split}
&\E \biggl[\bigg\vert \int_{0}^{\epsilon} \langle G(r,x-\ast), \Gamma(r,x-\ast) \rangle_{\mathcal{H}}  dr\bigg\vert^p \biggr] \\
&\leq c_{p} \bigg\vert \int_{0}^{\epsilon} \int_{\R^d} \int_{\R^d}  \Gamma(r,x-dy) (\E \, [\Vert u(t-r,y)-u(t, x)\Vert^p ])^{1/p} \Gamma(r,x-dz) f(y-z)dr\bigg\vert^p.
\end{split}
\end{equation*}
Therefore, appealing to hypotheses {\bf (H2)} and {\bf (H3)}, we
find that the last term is bounded by
\begin{equation*} \begin{split}
&c_{p,T} \biggl\{ \bigg\vert \int_{0}^{\epsilon} \int_{\R^d} \int_{\R^d}
\Gamma(r,x-dy) \Vert x-y \Vert^{\gamma_2} \Gamma(r,x-dz) f(y-z) dr \bigg\vert^p \\
& +\bigg\vert \int_{0}^{\epsilon} \int_{\R^d} \int_{\R^d}  \Gamma(r,x-dy) r^{\gamma_1} \Gamma(r,x-dz) f(y-z) dr \biggr\vert^p \biggr\} \\
&= c_{p,T} \biggl\{ \bigg\vert \int_{0}^{\epsilon} \langle
\Psi(r,\ast), \Gamma(r,\ast) \rangle_{\mathcal{H}}
dr\bigg\vert^p+\bigg\vert \int_{0}^{\epsilon} r^{\gamma_1}
\int_{\R^d} \vert \mathcal{F} \Gamma(r)(\xi) \vert^2  \mu(d\xi)
dr \biggr\vert^p \biggr\}.
\end{split}
\end{equation*}
Hence, using hypothesis  {\bf (H5)}, we conclude that
\begin{equation*}
\E \biggl[\bigg\vert \int_{0}^{\epsilon} \langle G(r,x-\ast), \Gamma(r,x-\ast) \rangle_{\mathcal{H}}  dr\bigg\vert^p \biggr] 
\leq c_{p,T} (\epsilon^{\alpha_1 p}+\epsilon^{\alpha_2 p}) \leq c_{p,T} \epsilon^{\alpha p}.
\end{equation*}
\end{proof}

Next, we recall some elements of Malliavin calculus.
Consider the Gaussian family $\{W(h), h  \in \mathcal{H}^q_T\}$
defined in Section 2, that is, a centered Gaussian process such that
$\E[W(h)W(g)]=\langle h,g \rangle_{\mathcal{H}^q_T}$. Then, we can use
the differential Malliavin calculus based on it (see for instance
\cite{Nualart:06}). We denote the Malliavin derivative by
$D=(D^{(1)},...,D^{(q)})$, which is an operator in $L^2(\Omega;
\mathcal{H}^q_T)$. For any $m \geq 1$, the domain of the iterated
derivative $D^m$ in $L^p(\Omega; (\mathcal{H}^q_T)^{\otimes m})$ is
denoted by $\mathbb{D}^{m,p}$, for any $p \geq 2$. We set
$\mathbb{D}^{\infty}= \cap _{p \geq 1} \cap_{m \geq 1}
\mathbb{D}^{m,p}$. Recall that for any differentiable random
variable $F$ and any $r=(r_1,...,r_m) \in [0,T]^m$, $D^m F(r)$ is an
element of $\mathcal{H}^{\otimes m}$, which will be denoted by
$D^m_r F$. We define the Malliavin matrix of $F \in
(\mathbb{D}^{\infty})^{m}$ by $\gamma_F=(\langle DF_i, D F_j
\rangle_{\mathcal{H}^q_T})_{1\leq i,j\leq m}$.

The next result is the $q$-dimensional extension of
\cite[Proposition 6.1]{Nualart:07} (see also \cite[Theorem 1]{Sanz:04}). Its proof follows exactly along
the same lines working coordinate by coordinate, and is therefore
omitted.
\begin{prop} \label{derivat}
Assume that {\bf (H1)} holds, and that $\sigma, b$ are $\mathcal{C}^{\infty}$ functions with bounded partial derivatives of order greater than or equal to one.
Then, for every $(t,x) \in (0,T] \times \R^d$, the random variable
$u_i(t,x)$ belongs to the space $\mathbb{D}^{\infty}$, for all
$i=1,...,k$. Moreover, the derivative $D u_i(t,x)=(D^{(1)}
u_i(t,x),...,D^{(q)} u_i(t,x))$ is an $\mathcal{H}^q_T$-valued
process that satisfies the following linear stochastic differential
equation, for all $i=1,...,k$:
\begin{equation} \label{eqder}
 \begin{split}
 D^{(j)}_r u_i(t,x) &= \sigma_{ij}(u(r, \ast)) \Gamma(t-r, x-\ast) \\
&+\int_r^t \int_{\R^d} \Gamma (t-s, x-y) \sum_{\ell=1}^{q} D^{(j)}_r (\sigma_{i \ell}(u(s,y)) W^{\ell}(ds, dy) \\
&+\int_r^t \int_{\R^d} \Gamma (t-s, dy) D^{(j)}_r (b_i(u(s,x-y)) ds,
 \end{split}
\end{equation}
for all $r\in [0,t]$, and is $0$ otherwise. Moreover, for all $p
\geq 1$, $m\geq 1$ and $i=1,...,k$, it holds that
\begin{equation} \label{num2}
\sup_{(t,x) \in (0,T] \times \R^d} \E \biggl[\Vert D^m u_i(t,x)
\Vert^p_{(\mathcal{H}^{q}_T)^{\otimes m}} \biggr] < +\infty
\end{equation}
\end{prop}

\begin{remark}
Recall that the iterated derivative of $u_i(t,x)$ satisfies also the
$q$-dimensional extension of equation \cite[(6.27)]{Nualart:07}.
\end{remark}

Recall that the stochastic integral on the right hand-side of (\ref{eqder}) must be understood by means of a Hilbert-valued
stochastic integral (see \cite[Section 3]{Nualart:07}) and the Hilbert-valued pathwise integral of (\ref{eqder})
is defined in \cite[Section 5]{Nualart:07}.
\vskip 12pt

In order to prove Theorem \ref{maint0}, we will use the following localized variant of Malliavin's absolute continuity theorem.
\begin{thm} \textnormal{\cite[Theorem 3.1]{Bally:98}} \label{loc}
Let $\Sigma_m \subset \R^{k}$, $m \in \mathbb{N}$, $m \geq1$, be a
sequence of open sets such that $\bar{\Sigma}_m \subset
\Sigma_{m+1}$ and let $F \in (\mathbb{D}^{\infty})^{k}$ such that
for any $q > 1$ and $m \in \mathbb{N}$,
\begin{equation} \label{expe}
\E  \bigl[( \textnormal{det} \, \gamma_F)^{-q} \1_{\{F \in \Sigma_m\}}\bigr] < +\infty.
\end{equation}
Then the law of $F$ admits a $\mathcal{C}^{\infty}$ density on the set $\Sigma=\cup_m \Sigma_m$.
\end{thm}

\noindent {\it Proof of Theorem \ref{maint0}.}
For each $m \in \mathbb{N}$, $m \geq 1$, define the open set
\begin{equation} \label{ga}
\Sigma_m=\biggl\{y \in \R^{k}: \sum_{j=1}^q \langle \sigma_j(y), \xi \rangle^2 > \frac{1}{m}, \, \forall \, \xi \in \mathbb{R}^{k}, \Vert
\xi \Vert=1\biggr\}.
\end{equation}
By Proposition \ref{derivat}, it suffices to prove that condition
(\ref{expe}) of Theorem \ref{loc} holds true in each $\Sigma_m$.
Indeed, observe that $\Sigma=\cup_m \Sigma_m=\{\sigma_1,...,\sigma_q \text{ span } \R^k \}$.

Let $(t,x) \in ]0,T] \times \R^d$ and $m \geq 1$ be fixed. It suffices to prove that there exists
$\delta_0(m)>0$ such that for all $0<\delta \leq
\delta_0$, and all $p>1$,
\begin{equation} \label{conclusion}
\P \biggl\{  (\textnormal{det} \, \gamma_{u(t,x)}< \delta) \1_{\{u(t,x) \in \Sigma_m\}} \biggr\} \leq C \delta^{\lambda p},
\end{equation}
for some $\lambda>0$, and for some constant $C>0$ not depending on $\delta$. This implies (\ref{expe}), taking $p=\frac{q}{\lambda}+1$
in (\ref{conclusion}).

We write
\begin{equation} \label{det}
\textnormal{det} \, \gamma_{u(t,x)} \geq \biggl(
\textnormal{inf}_{\xi \in \mathbb{R}^{k}:  \Vert \xi \Vert =1}
(\xi^{T} \gamma_{u(t,x)} \xi)\biggr)^{k}.
\end{equation}
Let $\xi \in \mathbb{R}^{k}$ with $\Vert \xi \Vert=1$, and fix
$\epsilon \in (0,1]$. The inequality
\begin{equation*} 
\Vert a+b \Vert_{\mathcal{H}^{q}}^2 \geq \frac{2}{3} \Vert a
\Vert_{\mathcal{H}^{q}}^2-2 \Vert b \Vert_{\mathcal{H}^{q}}^2,
\end{equation*}
together with (\ref{eqder}), gives
\begin{equation*}
\xi^{T}  \gamma_{u(t,x)}  \xi \geq \int_{t-\epsilon}^t
 \bigg\Vert \sum_{i=1}^{k} D_{r}(u_i(t,x)) \xi_i \bigg\Vert_{\mathcal{H}^{q}}^2 dr \geq  \frac{2}{3} \mathcal{A}_1 -2 \mathcal{A}_2,
\end{equation*}
where
\begin{equation*} \begin{split}
\mathcal{A}_{1}&=\sum_{j=1}^{q} \int_{t-\epsilon}^t
 \Vert \langle \sigma_{j}(u(r, \ast)), \xi \rangle  \Gamma(t-r, x-\ast)\Vert_{\mathcal{H}^{q}}^2 dr, \\
\mathcal{A}_2&=\int_{t-\epsilon}^t  \bigg\Vert \sum_{i=1}^{k} a_i(r,t,x,\ast) \,\xi_i \bigg\Vert_{\mathcal{H}^{q}}^2 dr, \\
a_i(r,t,x, \ast)&=  \sum_{\ell=1}^{q}\int_r^t \int_{\R^d} \Gamma (t-s, x-y) D_r (\sigma_{i \ell}(u(s,y)) W^{\ell}(ds, dy) \\
&\qquad +\int_r^t \int_{\R^d}  D_r (b_i(u(s,x-y)) \Gamma (t-s, y) dy ds.
\end{split}
\end{equation*}

Now, assume that $u(t,x) \in \Sigma_m$. Then,
adding and subtracting the term 
$$\mathcal{A}_{1,1}=\sum_{j=1}^{q} \langle \sigma_j(u(t,x)), \xi \rangle^2 \int_{t-\epsilon}^{t}  \int_{\R^d}
\vert \mathcal{F} \Gamma (t-r)(\xi) \vert^2 \mu(d\xi) dr$$ 
we get that $\mathcal{A}_1 \geq 
\mathcal{A}_{1,1} -\vert \mathcal{A}_{1,2} \vert$, where
\begin{equation*} \begin{split}
\mathcal{A}_{1,2}&= \sum_{j=1}^{q} \int_{t-\epsilon}^t \int_{\R^d}
\int_{\R^d} \Gamma(t-r,x-dy) \Gamma(t-r,x-dz) f(y-z)  \\
&\qquad \qquad \times \biggl( \langle \sigma_j(u(r,y)), \xi \rangle
\langle \sigma_j(u(r,z)), \xi \rangle- \langle \sigma_j(u(t,x)), \xi \rangle^2\biggr)\, dr.
\end{split}
\end{equation*}
Note that we have added and subtracted a "local" term to make the ellipticity property  appear (see (\ref{ga})). A similar idea is used in
\cite{Millet:99} for the stochastic wave equation in dimension 2.

Then, using the fact that $u(t,x) \in \Sigma_m$, we get that
\begin{equation*}
\mathcal{A}_{1,1} > \frac{1}{m} \int_0^{\epsilon} \int_{\R^d} \vert \mathcal{F}\Gamma(r) (\xi) \vert^2 \mu(d \xi) dr
=: \frac{1}{m} g(\epsilon).
\end{equation*}

Now we find out upper bounds for the $p$-th moment, $p > 1$, of the
terms $\mathcal{A}_{1,2}$ and $\mathcal{A}_{2}$. We start treating
$\mathcal{A}_{1,2}$. We write
\begin{equation*} \begin{split}
&\langle \sigma_j(u(r,y)), \xi \rangle
\langle \sigma_j(u(r,z)), \xi \rangle- \langle \sigma_j(u(t,x)), \xi \rangle^2\\
&=\langle \sigma_j(u(r,y)), \xi \rangle \, \langle \sigma_j(u(r,z))-\sigma_j(u(t,x)), \xi \rangle  \\
& \qquad \qquad \qquad +\langle \sigma_j(u(t,x)), \xi \rangle\,
\langle \sigma_j(u(r,y))-\sigma_j(u(t,x)), \xi \rangle.
\end{split}
\end{equation*}
Then, proceeding as in the proof of Lemma \ref{mes}, using the
Lipschitz property of the coefficients of $\sigma$, and
(\ref{supu}), it yields that
\begin{equation*}
\E \,  \biggl[\sup_{\xi \in \R^d: \Vert \xi \Vert=1} \vert
\mathcal{A}_{1,2} \vert^p \biggr] \leq c_{p,T} \epsilon^{\alpha
p}.
\end{equation*}

We next treat $\mathcal{A}_{2}$. Using the Cauchy-Schwarz
inequality, for any $p > 1$, it yields that
\begin{equation} \begin{split} \label{num1}
&\E \,  \biggl[\sup_{\xi \in \R^d: \Vert \xi \Vert=1}
\vert \mathcal{A}_2 \vert^p \biggr] \leq c_p \sum_{i=1}^{k} \E \,  \biggl[ \bigg\vert \int_{t-\epsilon}^t   \Vert a_i(r,t,x,\ast) \Vert_{\mathcal{H}^{q}}^2 dr \bigg\vert^p  \biggr] \\
& \qquad \leq c_p \sum_{i=1}^{k} \biggl(\E \,  \biggl[ \bigg\vert
\int_{0}^{\epsilon}  \Vert V_i(r,t,x,\ast) \Vert_{\mathcal{H}^{q}}^2
dr \bigg\vert^p  \biggr] + \E \,  \biggl[ \bigg\vert
\int_{0}^{\epsilon}   \Vert W_i(r,t,x,\ast)
\Vert_{\mathcal{H}^{q}}^2 dr \bigg\vert^p  \biggr]\biggr),
\end{split}
\end{equation}
where
\begin{equation*} \begin{split}
V_i(r,t,x,\ast)&:= \sum_{\ell=1}^{q}\int_{t-r}^t \int_{\R^d} \Gamma
(t-s, x-y)  D_{t-r} (\sigma_{i \ell}(u(s,y))W^{\ell}(ds,
dy),\\W_i(r,t,x,\ast)&:=\int_{t-r}^t \int_{\R^d} \Gamma (t-s, y)
D_{t-r} (b_i(u(s,x-y)) dy ds.
\end{split}
\end{equation*}
Now, using H\"older's inequality, the boundedness of the
coefficients of the derivatives of $\sigma$, and \cite[(3.13),
(5.26)]{Nualart:07}, we get that
\begin{equation*} \begin{split}
\E \,  \biggl[ \bigg\vert \int_{0}^{\epsilon}   \Vert
V_i(r,t,x,\ast) \Vert_{\mathcal{H}^{q}}^2 dr \bigg\vert^p \biggr]
&\leq c_p \epsilon^{p-1} \sup_{(s,y) \in (0,\epsilon) \times \R^d} \E[\Vert D_{t-\cdot} u_i(t-s,y) \Vert^{2p}_{\mathcal{H}_{\epsilon}^{q}}] \\ & \qquad \qquad \qquad \qquad \times \biggl(\int_0^{\epsilon} \int_{\R^d} \vert \mathcal{F}\Gamma(s) (\xi) \vert^2 \mu(d \xi) ds\biggr)^{p} \\
& \leq c_p g(\epsilon)^{2p}.
\end{split}
\end{equation*}
Moreover, using H\"older's inequality, the boundedness of the
partial derivatives of the coefficients of $b$, hypothesis
(\ref{hyp0}), and \cite[(5.17), (5.26)]{Nualart:07},
for the second term in (\ref{num1}) corresponding to the
Hilbert-valued pathwise integral, we obtain that
\begin{equation*} \begin{split}
\E \,  \biggl[ \bigg\vert \int_{0}^{\epsilon} \Vert W_i(r,t,x,\ast)
\Vert_{\mathcal{H}^{q}}^2 dr \bigg\vert^p \biggr]
&\leq c_p \epsilon^{p-1} \sup_{(s,y) \in (0,\epsilon) \times \R^d} \E[\Vert D_{t-\cdot} u_i(t-s,y) \Vert^{2p}_{\mathcal{H}_{\epsilon}^{q}}] \\
& \qquad \qquad \qquad \qquad \times\biggl(\int_0^{\epsilon} \int_{\R^d} \Gamma(s, y) dyds \biggr)\\
& \leq c_p \epsilon^{p}g(\epsilon)^{p}.
\end{split}
\end{equation*}
Hence, we conclude that, for any $p > 1$,
\begin{equation*}
\E \,  \biggl[\sup_{\xi \in \R^d: \Vert \xi \Vert=1} \vert \mathcal{A}_2 \vert^p \biggr] \leq c_p (g(\epsilon)^{2p}+ \epsilon^{p}g(\epsilon)^{p}).
\end{equation*}

Appealing to (\ref{det}), we have proved that, on the set $\{u(t,x) \in \Sigma_m\}$,
\begin{equation} \label{det2}
\textnormal{det} \, \gamma_{u(t,x)} > \biggl(\frac{2}{3 m} g(\epsilon)-I \biggr)^{k},
\end{equation}
where $I$ is a random variable such that for all $p>1$, $\E[\vert I
\vert^p ]\leq (\epsilon^{\alpha p}+g(\epsilon)^{2p}+ \epsilon^{p}g(\epsilon)^{p})$.

We now choose $\epsilon=\epsilon(\delta,m)$ in such a way
that $\delta^{1/k}=\frac{1}{3m} g(\epsilon)$. By hypothesis {\bf (H4)} this implies that $3m \delta^{1/k}\geq c \epsilon^{\eta}$, that is, $\epsilon \leq C \delta^{\frac{1}{\eta k}}$. Then using (\ref{det2}), we
conclude that for all $0<\delta \leq \delta_0$ and $p>1$,
$$
 \P \biggl\{  (\textnormal{det} \, \gamma_{u(t,x)} < \delta) \1_{\{u(t,x) \in \Sigma_m\}} \biggr\} \leq
\frac{\epsilon^{\alpha p}+g(\epsilon)^{2p}+ \epsilon^{p}g(\epsilon)^{p}}{(\frac{2}{3 m} g(\epsilon)-\delta^{1/k})^p} 
\leq C(m) (\delta^{\frac{p}{k}}+\delta^{\frac{p}{\eta k}}+
\delta^{\lambda p}),
$$
with $\lambda=\frac{\alpha-\eta}{\eta k}$. Recall that $\eta<\alpha$ from hypothesis {\bf (H6)}. This
proves (\ref{conclusion}). \hfill $\square$ \vskip 12pt 

We end this section with the verification of condition {\bf (H6)} for the stochastic heat and wave equations
with a Riesz kernel as spatial homogeneous covariance, that is, $f(x)=\Vert x \Vert^{-\beta}$ and $\mu(d\xi)=c_{d,\beta} \Vert \xi \Vert^{d-\beta} d \xi$, $0<\beta <(2 \wedge d)$. 

Consider first the stochastic heat equation in any spatial dimension, that is, 
$$\Gamma(t,x)=(4 \pi t)^{-d/2} \exp \biggl(-\frac{\Vert x \Vert^2}{4t} \biggr), \qquad t \geq 0, x \in \R^d.$$
For all $t \in [0,T]$, using the change variables
$[\tilde{y}=\frac{y}{\sqrt{r}},\ \tilde{z}=\frac{z}{\sqrt{r}}]$,
it yields that
\begin{equation*} \begin{split}
 &\int_0^t
 \int_{\R^d}
\int_{\R^d}  \Vert y-z \Vert^{-\beta} \Gamma(r,y)
\Gamma(r,z) \, dy dz dr \\
& \qquad = \int_0^{t}  r^{-\beta/2}
 \int_{\R^d} \int_{\R^d}   \Vert \tilde{y}- \tilde{z} \Vert^{-\beta} \Gamma(1,\tilde{y})
\Gamma(1,\tilde{z}) d\tilde{y} d\tilde{z} dr = c t^{\frac{2-\beta}{2}}.
\end{split}
\end{equation*}
Thus, hypothesis {\bf (H4)} is satisfied for
$\eta=\frac{2-\beta}{2}$.
Furthermore, hypothesis {\bf (H5)} holds with
$\alpha_1=\frac{2-\beta}{2}+\frac{\gamma_2}{2}$ and
$\alpha_2=\frac{2-\beta}{2}+\gamma_1$. Indeed, using the
change of variables $[\tilde{y}=\frac{y}{\sqrt{r}},
\tilde{z}=\frac{z}{\sqrt{r}}]$, for all $\tau \in [0,1]$, we
have that
\begin{equation*} \begin{split}
 & \int_{0}^{\tau} \int_{\R^d} \int_{\R^d}  \,\Vert z \Vert^{\gamma_2}  \Vert y-z \Vert^{-\beta} \Gamma(r,z)
\Gamma(r,y) dz dy dr \\
& =\int_{0}^{\tau} r^{\frac{\gamma_2}{2}-\frac{\beta}{2}}
\int_{\R^d} \int_{\R^d}  \Vert \tilde{z} \Vert^{\gamma_2} \Vert
\tilde{y}-\tilde{z} \Vert^{-\beta} \Gamma(1,\tilde{z})
\Gamma(1,\tilde{y}) d\tilde{z} d\tilde{y} dr =c
\tau^{\frac{2-\beta}{2}+\frac{\gamma_2}{2}},
      \end{split}
\end{equation*}
and
\begin{equation*} \begin{split}
 & \int_{0}^{\tau} r^{\gamma_1} \int_{\R^d} \int_{\R^d}   \Vert y-z \Vert^{-\beta} \Gamma(r,z)
\Gamma(r,y) dz dy dr \\
& =\int_{0}^{\tau} r^{\gamma_1-\frac{\beta}{2}} \int_{\R^d}
\int_{\R^d} \Vert \tilde{y}-\tilde{z} \Vert^{-\beta}
\Gamma(1,\tilde{z}) \Gamma(1,\tilde{y}) d\tilde{z} d\tilde{y} dr =c
\tau^{\frac{2-\beta}{2}+\gamma_1}.
      \end{split}
\end{equation*}
Observe that $\alpha=\alpha_1 \wedge \alpha_2 =
\frac{2-\beta}{2}+\frac{\gamma_2 \wedge (2 \gamma_1) }{2}$, which is bigger that $\eta$, thus proves {\bf (H6)}. 

We next treat the case of the stochastic wave equation with spatial dimension $d \in \{1,2,3\}$. Let $\Gamma_d$ be the fundamental solution of the deterministic wave
equation in $\R^d$ with null initial conditions.
In this case, changing variables  $[\tilde{\xi}=r \xi]$, for all $t \in [0,T]$, we get that
\begin{equation*} 
\begin{split}
\int_0^{t} \int_{\R^d} \vert \mathcal{F}\Gamma_d(r) (\xi) \vert^2
\Vert \xi \Vert^{\beta-d} \, d \xi dr &=\int_0^{t} \int_{\R^d} \frac{\sin^2(2\pi r \Vert \xi \Vert)}{4 \pi^2 \Vert \xi \Vert^2}
\Vert \xi \Vert^{\beta-d} \, d \xi dr \\
&=\int_0^{t} r^{2-\beta} \int_{\R^d} \frac{\sin^2(2\pi \Vert \tilde{\xi} \Vert)}{4 \pi^2 \Vert \tilde{\xi} \Vert^2}
\Vert \tilde{\xi} \Vert^{\beta-d} \, d \tilde{\xi} dr \\
&=c t^{3-\beta}.
\end{split}
\end{equation*}
Therefore, hypothesis {\bf (H4)} is satisfied for
$\eta=3-\beta$.
Notice also that
\begin{equation*}
\int_0^T \int_{\R^d} \vert \mathcal{F}\Psi(r) (\xi) \vert^2 \Vert
\xi \Vert^{\beta-d} \, d \xi dr \leq T^{2\gamma_2} \int_0^T
\int_{\R^d} \vert \mathcal{F}\Gamma_d(r) (\xi) \vert^2 \Vert \xi
\Vert^{\beta-d} \, d \xi dr <+\infty .
\end{equation*}
Moreover, changing variables $[\tilde{\xi}=r \xi]$, for all $\tau \in [0,1]$, we have that
\begin{equation*} \begin{split}
& \int_{0}^{\tau} \langle \Psi(r,\ast), \Gamma_d(r,\ast)
\rangle_{\mathcal{H}}  dr =\int_{0}^{\tau} \int_{\R^d} \mathcal{F}\Psi(r) (\xi)
\frac{\sin(2\pi r \Vert \xi \Vert)}{2 \pi \Vert \xi \Vert} \Vert \xi \Vert^{\beta-d} \, d \xi  dr \\
&\qquad =\int_{0}^{\tau} r^{1-\beta} \int_{\R^d} \mathcal{F}\Psi(r) \biggl(\frac{\tilde{\xi}}{r}\biggr)
\frac{\sin(2\pi \Vert \tilde{\xi} \Vert)}{2 \pi \Vert \tilde{\xi} \Vert} \Vert \tilde{\xi} \Vert^{\beta-d} \, d \tilde{\xi}  dr \\
&\qquad =  \langle \Psi(1,\ast), \Gamma_d(1,\ast) \rangle_{\mathcal{H}}
\int_{0}^{\tau} r^{2-\beta+\gamma_2}\, dr =c
\tau^{3-\beta+\gamma_2},
\end{split}
\end{equation*}
and
\begin{equation*}
\int_0^{\tau} r^{\gamma_1} \int_{\R^d} \vert \mathcal{F}\Gamma_d(r)
(\xi) \vert^2 \Vert \xi \Vert^{\beta-d} \, d \xi dr =c
\tau^{3-\beta+\gamma_1}.
\end{equation*}
Thus, hypotheses {\bf (H5)} and {\bf (H6)} hold taking $\alpha=3-\beta +\gamma_1
\wedge \gamma_2>3-\beta=\eta$.

\section{Strict positivity of the density}

The second aim of this article is to show that under
the same condition {\bf (H6)}, the density of the law
of the solution to the system (\ref{equa1}) is strictly positive in a point if the connected
component of the set where $\{\sigma_1,...,\sigma_q\}$ span $\R^k$ that contains this point has a non void intersection with the support of the law of the solution 
(see Theorem \ref{maintt} below). For this, we will first extend in our
situation a criterion of strict positiveness of densities proved by
Bally and Pardoux in \cite{Bally:98}, which uses essentially an
inverse function type result (see Lemma \ref{inv}) and Girsanov's
theorem. 

We will then apply this criterion to our system of SPDEs
(\ref{equa1}). In \cite{Bally:98}, the authors apply their criterion
to the density of the law of the random vector
$(u(t,x_1),....,u(t,x_k))$, $0 \leq x_1 \leq \cdots \leq x_k \leq
1$, where $u$ denotes the solution to the non-linear stochastic heat
equation driven by a space-time white noise, by using a localizing
argument. Hence, their situation is a bit different as ours, as we
deal with a system of SPDEs and we evaluate the solution at a single
point $(t,x) \in ]0,T] \times \R^d$. In a similar context, Chaleyat
and Sanz-Sol\'e in \cite{Chaleyat:03} study the strict positivity of
the density of the random vector $(u(t,x_1),....,u(t,x_k))$, $0 \leq
x_1 \leq \cdots \leq x_k \leq 1$, where $u$ is the solution to the
stochastic wave equation in two spatial dimensions, that is, take in
equation (\ref{equa1}) $d=2$, $k,q=1$, and $L=\frac{\partial^2}{\partial t^2}-\Delta$.

We will also apply the criterion of strict positivity 
to the case of a single equation, that is the solution
to (\ref{equa1}) with $k,q=1$. We will obtain that the density is strictly positivity 
in all $\R$ under the same hypotheses that D.Nualart and Quer-Sardanyons obtained in
\cite{Nualart:07} its existence and smoothness (see Theorem \ref{tdim1} below), with the addition condition of $\sigma$ being bounded.
Recall that in the recent papers \cite{Nualart:09} and \cite{Nualart:10}, the same authors obtain lower and upper bounds of Gaussian
type for the density of the solution to single spatially homogeneous SPDEs of the class (\ref{equa1}) in the case where $\sigma$ is a constant, which 
in particular shows the
strict positivity of the density in the quasi-linear case. A first step to show bounds for the non-linear case is done in the recent preprint \cite{ENualart:10}, where upper and lower bounds
for the density of non-linear stochastic heat equations in any space dimension are established.

There are also other many SPDEs for which the strict positivity of its density is studied.
For example, the case of non-linear hyperbolic SPDEs has been studied by Millet and Sanz-Sol\'e in \cite{Millet:97}, Fournier \cite{Fournier:99} considers a Poisson driven SPDE,
and the Cahn-Hilliard stochastic equation is studied by Cardon-Weber in \cite{Cardon:02}.

We next state the main result of this paper.
\begin{thm} \label{maintt}
Assume hypothesis {\bf (H6)} and that $\sigma, b$ are $\mathcal{C}^{\infty}$ functions with bounded partial derivatives of order greater than or equal to one and $\sigma$ is bounded.
Then for all $(t,x) \in
]0,T] \times \R^d$, the law of the random vector $u(t,x)$ admits a
$\mathcal{C}^{\infty}$ density on $\Sigma:=\{y \in \R^{k}: \sigma_1(y),...,\sigma_q(y) \text{ span } \R^k\}$ such that 
$p_{t,x}(y)>0$ if the connected component of $\Sigma$ which contains $y$ has a non void interesection with the support
of $u(t,x)$.
\end{thm}

\begin{remark} \label{ue}
In the case where $\sigma$ is uniformly elliptic, that is, $\Vert \sigma(y) \xi \Vert^2 \geq c >0$, for all $\xi \in \R^q$ and $y \in \R^k$, under the hypotheses 
of Theorem \ref{maintt}, we get that for all $(t,x) \in ]0,T] \times \R^d$, the law of the random vector $u(t,x)$ admits a
$\mathcal{C}^{\infty}$ strictly positive density.
\end{remark}
  
The case of a single SPDE of the type (\ref{equa1}) is as follows.
\begin{thm} \label{tdim1}
Assume that $\sigma, b$ are $\mathcal{C}^{\infty}$ functions with 
bounded derivatives of order greater than or equal to one, and $0<c \leq \vert \sigma \vert \leq C$. 
Then, under conditions {\bf (H1)} and {\bf (H4)}, 
for all $(t,x) \in ]0,T] \times \R^d$, the law of $u(t,x)$ has a $\mathcal{C}^{\infty}$ strictly positive density function.
\end{thm}

\begin{remark}
Observe that we only require
the continuity of the density in order to show the strict positivity (see Theorem \ref{tm1} below).
Hence the $\mathcal{C}^{\infty}$ condition on the coefficients
can be weakened in order to show the strict positivity of the density.
\end{remark}

\subsection{Criterion for the strict positivity of the density}

As in Section $4$, we will study the strict positivity of the density on the set where the columns of the $k \times q$ matrix $\sigma$,
$\sigma_1,...,\sigma_q$ span $\R^k$. If $y \in \R^k$ is in this set, then there exists $l_1(y)=:l_1,...,l_k(y)=:l_k$
such that $\sigma_{l_1}(y),...,\sigma_{l_k}(y)$ span $\R^k$. Thus it suffices to make all the calculations using 
$W^{l_1(y)},...,W^{l_k(y)}$ and ignoring the other $W^j$. In order to simplify the notation we will take $l_i=i$.

Given $T>0$, a predictable process $g  \in \mathcal{H}_T^k$ and $z \in \R^{k}$, we define the process
$\hat{W}=(\hat{W}^1,...,\hat{W}^{k})$ as
$$
\hat{W}^j (1_{[0,t]}h_j)=W^j(1_{[0,t]}h_j) + z_j \int_0^t \langle
h_j(\ast), g^{j}(s,\ast) \rangle_{\mathcal{H}}  ds, $$
for any $h \in \mathcal{H}^k$, $j=1,...,k$, $t \in [0,T]$.

We set $\hat{W}_t(h)=\sum_{j=1}^{k} \hat{W}^j (1_{[0,t]} h_j)$.
Then, by \cite[Theorem 10.14]{Daprato:92}, $\{\hat{W}_t, t \in
[0,T]\}$ is a cylindrical Wiener process in $\mathcal{H}^{k}$ on the
probability space $(\Omega, \mathcal{F}, \hat{\P})$, where
$$
\frac{d \hat{\P}}{d \P} (\omega)= J_z(\omega), \; \; \omega \in
\Omega,
$$
where
\begin{equation} \label{Girsanov}
J_{z}=\exp\biggl(- \sum_{j=1}^{k} z_j \int_0^T \int_{\R^d} g^j(s,y)
W^j(ds, dy)-\frac12 \sum_{j=1}^{k} z_j^2 \int_0^T \Vert g^j(s,\ast)
\Vert_{\mathcal{H}}^2 \,ds\biggr).
\end{equation}
Then, for any predictable process $Z \in L^2(\Omega \times [0,T];
\mathcal{H}^{k})$ and $j=1,...,k$, it yields that
\begin{equation*} \begin{split}
\int_0^T \int_{\R^d} Z_{j}(s,y) \hat{W}^j(ds, dy)
=\int_0^T \int_{\R^d} Z_{j}(s,y) W^j(ds, dy) +z_j \int_0^T \langle Z_j(s, \ast), g^j(s, \ast) \rangle_{\mathcal{H}}  ds.
\end{split}
\end{equation*}

For any $(t,x) \in [0,T] \times \R^k$, let $\hat{u}^z(t,x)$ be the solution to equation (\ref{equa2}) with
respect to the cylindrical Wiener process $\hat{W}$, that is, for
$i=1,...,k$,
\begin{equation} \label{direcder}
\begin{split}
\hat{u}_i^z(t,x)&= \sum_{j=1}^{k} \int_0^t \int_{\R^d} \Gamma(t-s, x-y)  \sigma_{ij}(\hat{u}^z(s,y)) W^j(ds, dy) \\
&+\sum_{j=1}^{k} z_j \int_0^t \langle \Gamma(t-s, x-\cdot) \sigma_{ij}(\hat{u}^z(s,\ast)), g^j(s,\ast) \rangle_{\mathcal{H}} ds \\
&+\int_0^t \int_{\R^d} b_i(\hat{u}^z(t-s,x-y)) \Gamma(s, dy) ds.
\end{split}
\end{equation}
Then, the law of $u$ under $\P$ coincides with the law of
$\hat{u}^z$ under $\hat{P}$, which implies that for all  
nonnegative continuous and bounded function $f:\R^{k} \mapsto \R$,
\begin{equation} \label{Girsanov2}
\E[f(u(t,x))]=\E[f(\hat{u}^z(t,x)) J_z].
\end{equation}

Given a sequence $\{ g_n \}_{n \geq1}$ of predictable processes in
$\mathcal{H}^k_T$ and $z \in \R^{k}$, let
$\hat{u}^z_n(t,x)$ be the solution to equation (\ref{equa2}) with
respect to the cylindrical Wiener process $\{\hat{W}^n_t, t \in
[0,T]\}$, where $\hat{W}^n_t(h)=\sum_{j=1}^{k} \hat{W}^{n,j}
(1_{[0,t]} h_j)$ for any $h \in \mathcal{H}^k$, and
$$
\hat{W}^{n,j}(1_{[0,t]} h_j)=W^j(1_{[0,t]} h_j)+z_j\int_0^t  \langle h_j(\ast), g^j_n(s,\ast) \rangle_{\mathcal{H}}  ds.
$$
Set the $k \times k$ matrix $\varphi^z_n(t,x):=\partial_z
\hat{u}^z_n(t,x)$, and the Hessian matrix of the random vector
$\hat{u}^z_n(t,x)$, $\psi^z_n(t,x):=\partial^2_z \hat{u}^z_n(t,x)$
(which is a tensor of order 3). We denote by $\Vert \cdot \Vert$ the
norm of a $n \times n$ matrix $A$ defined as
$$
\Vert A \Vert=\sup_{\xi \in \R^{n}, \Vert \xi \Vert=1} \Vert A \xi
\Vert.
$$

We next proceed to the study of the strict positivity of the density $p_{t,x}(\cdot)$ of the law of $u(t,x)$, where $(t,x) \in ]0,T] \times \R^d$ are fixed.
We need to introduce the following condition. We say that $y \in \R^k$ satisfies ${\bf H_{t,x}(y)}$ if
\begin{itemize}
\item[${\bf H_{t,x}(y)}$] there exists a
sequence of predictable processes $\{g_n\}_{n \geq 1}$ in
$\mathcal{H}_T^k$, and positive constants $c_1$, $c_2$, $r_0$ and
$\delta$ such that
\begin{itemize}
\item[\textnormal{(i)}] $\limsup_{n \rightarrow \infty} \P \, \biggl\{ (\Vert u(t,x)-y \Vert \leq r)
\cap ( \textnormal{det} \, \varphi^0_n(t,x) \geq c_1)\biggr\} >0, \, \, \text{for all } \, r \in ]0,r_0]$.
\item[\textnormal{(ii)}] $\lim_{n \rightarrow \infty} \P \, \biggl\{ \sup_{\Vert z \Vert \leq \delta} (\Vert  \varphi^z_n(t,x)\Vert + \Vert  \psi^z_n(t,x) \Vert ) \leq c_2\biggr\} =1$.
\end{itemize}
\end{itemize}

\begin{remark}
Observe that if a point $y$ verifies ${\bf H_{t,x}(y)}$ then automatically belongs to the support of the law of $u(t,x)$.
\end{remark}

We next provide the main result of this section which is a criterion
for the strict positivity of a density, which was proved in
\cite[Theorem 3.3]{Bally:98} for the case where
$\mathcal{H}^k_T$ is replaced by $L^2([0,1];\R^{k})$. Our case follows along
the same lines as theirs. This criterium uses the following quantitative
version of the classical inverse function theorem. Let $B(x;r)$
denote the ball of $\R^{k}$ with center $x$ and radius $r>0$.
\begin{lem}  \textnormal{\cite[Lemma 3.2]{Bally:98} and \cite[Lemma 4.2.1]{Nualart:98}}  \label{inv}
Let $\Phi:\R^{k} \mapsto \R^{k}$ be a $\mathcal{C}^2$ mapping such that for some constants
$\beta>1$ and $\delta>0$, 
\begin{equation*}
\vert \textnormal{det} \, \Phi'(0) \vert \geq \frac{1}{\beta} \; \text{ and } \;
\sup_{\Vert z \Vert \leq \delta} (\vert \Phi(z) \vert +\vert \Phi'(z) \vert + \vert \Phi''(z) \vert) \leq \beta.
\end{equation*}
Then there exists constants $R \in ]0,1]$ and
$\alpha >0$, such that $\Phi$
is a diffeomorphism from a neighborhood of $0$ contained in the ball $B(0;R)$ onto the ball $B(\Phi(0); \alpha)$.
\end{lem}

\begin{thm} \label{tm1}
Let $(t,x) \in ]0,T] \times \R^d$ and $y \in \R^k$ be such that ${\bf H_{t,x}(y)}$ holds true. Suppose that the law of
random vector $u(t,x)$ has a continuous density $p_{t,x}$ in a neighborhood of $y$. Then $p_{t,x}(y)>0$.
Moreover, if ${\bf H_{t,x}(y)}$ holds on $\textnormal{Supp}(\P_{u(t,x)}) \cap \Sigma$, with
$\Sigma=\{y \in \R^k:\sigma_1(y),...,\sigma_q(y) \text{ span } \R^k \},$ then for every connected subset $\tilde{\Sigma} \subset \Sigma$ such that
$\textnormal{Supp}(\P_{u(t,x)}) \cap \tilde{\Sigma}$ is non void, $p_{t,x}$ is a strictly positive function on $\tilde{\Sigma}$.
\end{thm}

\begin{proof}
Let $y \in \R^k$ satisfying ${\bf H_{t,x}(y)}$. 
Then, proceeding as in \cite[Theorem 3.3]{Bally:98}, using Lemma \ref{inv} with $\Phi(z)=\hat{u}^z_n(t,x)$, one can show that there exists $n$ sufficiently large such that for any continuous and bounded function $f:\R^{k} \mapsto \R_+$ it holds that
\begin{equation*} 
\E[f(u(t,x))] \geq \int_{B(y,\frac{\alpha}{2})} f(v) \theta_n^y(v) dv,
\end{equation*}
where $\alpha>0$ is the parameter of Lemma \ref{inv} and $\theta_n^y$ is a strictly positive continuous function on $B(y,\frac{\alpha}{2})$.
This proves the first statement of the theorem.  

We next assume that ${\bf H_{t,x}(y)}$ holds on $\textnormal{Supp}(\P_{u(t,x)}) \cap \Sigma$. It suffices to check that if $\tilde{\Sigma} \subset \Sigma$ is a connected component
of $\Sigma$ such that
$\textnormal{Supp}(\P_{u(t,x)}) \cap \tilde{\Sigma}$ is non void, then $\tilde{\Sigma} \subset \textnormal{Supp}(\P_{u(t,x)})$. Indeed, this implies that ${\bf H_{t,x}(y)}$ 
holds for all $y \in \tilde{\Sigma}$ and by the first part of the theorem one obtains that $p_{t,x}(y)>0$ and the theorem is proved. 

Suppose that $\tilde{\Sigma} \subset \textnormal{Supp}(\P_{u(t,x)})$ is false. Then one may find $x_1 \in \tilde{\Sigma}$ such that $x_1 \notin \textnormal{Supp}(\P_{u(t,x)})$.
Moreover, since $\textnormal{Supp}(\P_{u(t,x)}) \cap \tilde{\Sigma}$ is non void one may find $x_2 \in \textnormal{Supp}(\P_{u(t,x)}) \cap \tilde{\Sigma}$. Now, since $\tilde{\Sigma}$
is connected, one can find
a continuous curve $x(\lambda)$, $\lambda \in [0,1]$ contained in $\tilde{\Sigma}$ with $x(0)=x_2$ and $x(1)=x_1$. Take now 
$\lambda_{\ast}=\text{sup} \{ \lambda: x(\lambda) \in   \textnormal{Supp}(\P_{u(t,x)})\}$.
Since $x_1=x(1) \notin \textnormal{Supp}(\P_{u(t,x)})$ and the complementary of $\textnormal{Supp}(\P_{u(t,x)})$ is an open set, it follows that $\lambda_{\ast}<1$. Then we
have a sequence $\lambda_n \uparrow \lambda_{\ast}$ such that $x(\lambda_n) \in \textnormal{Supp}(\P_{u(t,x)})$ and we also know that $x(\lambda) \notin  \textnormal{Supp}(\P_{u(t,x)})$
for $\lambda_{\ast}<\lambda \leq 1$. This means that $x(\lambda_{\ast})$ is on the boundary of $\textnormal{Supp}(\P_{u(t,x)})$, and since this set is closed, we conclude that
$x(\lambda_{\ast}) \in \textnormal{Supp}(\P_{u(t,x)})$. Since $x(\lambda_{\ast})$ belongs also to $\tilde{\Sigma}$ we have that $p_{t,x}(x(\lambda_{\ast}))>0$. But this is contradictory
with the fact that $x(\lambda_{\ast})$ is on the boundary of $\textnormal{Supp}(\P_{u(t,x)})$.
\end{proof}

\begin{remark}
Observe that Theorem \ref{tm1} says that if a specific point $y_{\ast}$ is in the support of the law of $u(t,x)$ then $p_{t,x}$ is strictly positive on the connected component
of $\Sigma$ that contains $y_{\ast}$. However, this criterion based on the inverse function theorem does not give information about the support of the law. Nevertheless,
in order to prove that $p_{t,x}(y)>0$, we do not have to check that $y$ itself belongs to the support, but only that the connected component of $\Sigma$ which contains $y$ has a non void intersection 
with the support, and in particular, if we have ellipticity everywhere, then $\Sigma=\R^k$ and so the above condition is automatically verified (see Remark \ref{ue}).
\end{remark}

\begin{remark} 
There is a little mistake in page 56, lines 11-13 of \cite{Bally:98}. Indeed, the support is a closed set and their set
$\{ \varphi \neq 0\}^d$ is open, thus they may not intersect, so it is not always possible to choose a point $y$ in the boundary of the support and
such that $\varphi(y) \neq 0$. For example, if $\{ \varphi \neq 0\}=\{-1,1\}^c$ and the support is $[-1,1]$, in the boundary
of the support one has that $\varphi =0$. 
\end{remark}

\subsection{Proof of Theorem \ref{maintt}}

Let $(t,x) \in ]0,T] \times \R^{d}$ be fixed. As explained at the beginning of Section 5.1, let $y\in \textnormal{Supp}(\P_{u(t,x)})$ be fixed such that $\sigma_1(y),...,\sigma_k(y)$ span $\R^k$. 
Consider the sequence of predictable processes
$\{g_n\}_{n \geq 1}$ in $\mathcal{H}_T^k$, defined by
\begin{equation}
g^{j}_{n}(s, z)=v_n^{-1} 1_{[t-2^{-n},t]} (s) \Gamma(t-s,x-z), \; \;
n \geq 1, \; j=1,...,k,
\end{equation}
where
$$
v_n:=\int_{0}^{2^{-n}} \int_{\R^d} \vert \mathcal{F}\Gamma(r) (\xi) \vert^2  \mu(d \xi) dr.
$$
We are going to prove that assumptions (i) and (ii) of ${\bf H_{t,x}(y)}$ are satisfied for this sequence of predictable processes. Then, then second statement of Theorem \ref{tm1} will give the conclusion of Theorem \ref{maintt}.

We start by some preliminary computations. We write
\begin{equation*}
\varphi^z_{n,i,j}(t,x)=\int_0^t \langle D^{(j)}_r(\hat{u}^z_{n,i}(t,x)) , g^j_{n}(r, \ast) \rangle_{\mathcal{H}} dr, \; 1 \leq i,j \leq k,
\end{equation*}
which follows since the processes defined by the left hand-side and right hand-side satisfy the same equation,
and there is uniqueness of solution.
Then by (\ref{direcder}) and the stochastic differential equation satisfied by
the derivative (\ref{eqder}), we obtain that
\begin{equation*}
\varphi^z_{n,i,j}(t,x)= \mathcal{A}^z_{n,i,j}(t,x)+\mathcal{B}^z_{n,i,j}(t,x)+\mathcal{C}^z_{n,i,j}(t,x),
\end{equation*}
where
\begin{equation*} \begin{split}
&\mathcal{A}^z_{n,i,j}(t,x)= v_n^{-1} \int_0^{2^{-n}}
\langle \sigma_{ij}(\hat{u}^z_n(t-r,\ast)) \Gamma(r, x-\ast),  \Gamma(r, x-\ast) \rangle_{\mathcal{H}} dr, \\
&\mathcal{B}^z_{n,i,j}(t,x):=\int_0^t \biggl(\int_r^t \int_{\R^d} \Gamma(t-s, x-y) \sum_{\ell, m=1}^{k} \partial_m \sigma_{i \ell}(\hat{u}^z_n(s,y)) \\
& \qquad \qquad \qquad \qquad \qquad \qquad \times \langle D^{(j)}_r(\hat{u}^z_{n,m}(s,y)),g^j_{n}(r,\ast) \rangle_{\mathcal{H}} \hat{W}^{n,\ell}(ds, dy)\biggr) dr, \\
&\mathcal{C}^z_{n,i,j}(t,x):=\int_0^t \biggl(\int_r^t \int_{\R^d} \sum_{m=1}^{k} \partial_m b_i(\hat{u}^z_n(s,x-y)) \\
& \qquad \qquad \qquad \qquad \qquad \qquad \times\langle D^{(j)}_r(\hat{u}^z_{n,m}(s,x-y)) , g^j_{n}(r,\ast) \rangle_{\mathcal{H}} \Gamma(t-s, dy) ds\biggr) dr.
\end{split}
\end{equation*} 
Note that $g^j_{n}(r,\ast)=1_{[t-2^{-n},t]} (r) g^j_{n}(r,\ast)$, and that $D_r(\hat{u}^z_{n,i}(s,y))=0$ if
$s<r$. Hence, using these facts and Fubini's theorem, it yields that
\begin{equation*} \begin{split}
\mathcal{B}^z_{n,i,j}(t,x)&=\int_{t-2^{-n}}^t  \int_{\R^d} \Gamma(t-s, x-y)\sum_{\ell, m=1}^{k} \partial_k \sigma_{i \ell}(\hat{u}^z_n(s,y)) \\
& \qquad \qquad  \qquad \qquad  \times \biggl(\int_{t-2^{-n}}^s \langle D^{(j)}_r(\hat{u}^z_{n,m}(s,y)),g^j_{n}(r,\ast) \rangle_{\mathcal{H}}dr \biggr)  \hat{W}^{n,\ell}(ds, dy), \\
\mathcal{C}^z_{n,i,j}(t,x)&=\int_{t-2^{-n}}^t  \int_{\R^d} \sum_{m=1}^{k} \partial_m b_i(\hat{u}^z_n(s,x-y)) \\
& \qquad \qquad  \qquad \qquad  \times \biggl(\int_{t-2^{-n}}^s  \langle D^{(j)}_r(\hat{u}^z_{n,m}(s,x-y)) , g^j_{n}(r,\ast) \rangle_{\mathcal{H}} dr \biggr)  \Gamma(t-s, dy) ds.
\end{split}
\end{equation*}
Therefore, we have proved that
\begin{equation} \label{ec1}
\varphi^z_{n,i,j}(t,x)=\mathcal{A}^z_{n,i,j}(t,x)+\mathcal{D}^z_{n,i,j}(t,x)+\mathcal{E}^z_{n,i,j}(t,x)+\mathcal{C}^z_{n,i,j}(t,x),
\end{equation}
where
\begin{equation*} \begin{split}
\mathcal{D}^z_{n,i,j}(t,x) &:=\int_{t-2^{-n}}^t  \int_{\R^d}
\Gamma(t-s, x-y) \sum_{\ell, m=1}^{k} \partial_m \sigma_{i
\ell}(\hat{u}^z_n(s,y))
\varphi^z_{n,m,j}(s,y) W^{\ell}(ds, dy), \\
\mathcal{E}^z_{n,i,j}(t,x) &:= \sum_{\ell,m=1}^{k} z_{\ell}
\int_{t-2^{-n}}^t  \langle \Gamma(t-s, x-\ast)  \partial_m \sigma_{i
\ell}(\hat{u}^z_n(s,\ast)) 
\varphi^z_{n,m,j}(s,\ast), g^{\ell}_n(s,\ast) \rangle_{\mathcal{H}} ds, \\
\mathcal{C}^z_{n,i,j}(t,x)&=\int_{t-2^{-n}}^t  \int_{\R^d}
\sum_{m=1}^{k} \partial_m b_i(\hat{u}^z_n(s,x-y))
\varphi^z_{n,m,j}(s,x-y) \Gamma(t-s, dy)  ds.
\end{split}
\end{equation*}

Next we study upper bounds for the $p$-moments of the four terms on the right hand side of $\varphi^z_{n,i,j}(t,x)$.  
For this, we assume that $\Vert z \Vert \leq \delta$ for some $\delta>0$.
 
First observe that since $\sigma$ is bounded, 
\begin{equation} \label{ec2} 
\vert \mathcal{A}^z_{n,i,j}(t,x) \vert \leq K.
\end{equation}

Appealing to \cite[(3.11)]{Nualart:07}
and using the fact that the partial
derivatives of the coefficients of $\sigma$ are bounded, we get that
for all $p>1$,
\begin{equation} \label{ec3} 
\E [\vert \mathcal{D}^z_{n,i,j}(t,x) \vert^{p}]
\leq  c_p v_n^{p/2} \sum_{m=1}^{k}  \sup_{(s,y) \in [0,T] \times \mathbb{R}^d}  
\E [  \vert \varphi^z_{n,m,j}(s,y) \vert^{p} ].
\end{equation}

Using the Cauchy-Schwarz inequality and the fact that the
derivatives of $\sigma$ are bounded, it yields that for all $p>1$,
\begin{equation} \label{ec4}
\E [\vert \mathcal{E}^z_{n,i,j}(t,x) \vert^{p}]  \leq c_{p} \delta^p \sum_{m=1}^{k}  \sup_{(s,y) \in [0,T] \times
\mathbb{R}^d} \E [  \vert \varphi^z_{n,m,j}(s,y) \vert^{p}].
\end{equation}

We finally use Minkowski's inequality with respect to the finite measure
$\Gamma(s,dy)ds$, the boundedness of the partial derivatives of the
coefficients of $b$, and hypothesis (\ref{hyp0}), to see that for all $p>1$,
\begin{equation} \label{ec5}\begin{split}
\E [\vert \mathcal{C}^z_{n,i,j}(t,x) \vert ^{p}]
& \leq  c_p  \sum_{m=1}^{k}  \sup_{(s,y) \in [0,T] \times \mathbb{R}^d}  \E [  \vert \varphi^z_{n,m,j}(s,y) \vert^{p}] \biggl(\int_{0}^{2^{-n}}  \int_{\R^d}  \Gamma(s, dy)  ds\biggr)^p\\
& \leq c_{p,T} 2^{-np}  \sum_{m=1}^{k}  \sup_{(s,y) \in [0,T] \times
\mathbb{R}^d}  \E [  \vert \varphi^z_{n,m,j}(s,y) \vert^{p}].
\end{split}
\end{equation}

Now, introducing (\ref{ec2}), (\ref{ec3}), (\ref{ec4}) and (\ref{ec5}) into (\ref{ec1}), we get that for all $p > 1$,
\begin{equation} \label{ec6} 
\sum_{m=1}^{k}\E [\vert \varphi^z_{n,m,j}(t,x) \vert ^{p}] \leq K_{p}+c_{p,T}(v_n^{p/2} + \delta^p+ 2^{-np}) \sum_{m=1}^{k}  \sup_{(s,y) \in [0,T] \times \mathbb{R}^d}  \E [  \vert \varphi^z_{n,m,j}(s,y) \vert^{p}].
\end{equation}
Observe that, proceeding as in the proof of (\ref{num2}), one can show that
\begin{equation} \label{ec7}
 \sup_{\Vert z \Vert \leq \delta}\sup_{(s,y) \in [0,T] \times \mathbb{R}^d}  \E [  \vert \varphi^z_{n,i,j}(s,y) \vert^{p}] < \infty,
\end{equation}
as we have shown in (\ref{ec1}) that $\varphi^z_{n,i,j}(t,x)$ satisfies a linear equation with initial condition $\mathcal{A}^z_{n,i,j}(t,x)$,
which is bounded. Thus, choosing $n$ large and $\delta$ small such that $c_{p,T}(v_n^{p/2} + \delta^p+ 2^{-np})\leq \frac12$, we obtain from (\ref{ec6}) and (\ref{ec7}) that
\begin{equation} \label{ec8}  
\sup_{\Vert z \Vert \leq \delta} \sup_{(s,y) \in [0,T] \times \mathbb{R}^d} \E [  \vert  \varphi^z_{n,i,j}(s,y) \vert ^{p} ]\leq K_{p}.
\end{equation}

We are now ready to show that (i) and (ii) of ${\bf H_{t,x}(y)}$ are verified.
\vskip 12pt
\noindent {\it Proof of} (i). Take $z=0$ in (\ref{ec1}), and use (\ref{ec3}), (\ref{ec5}) and (\ref{ec8}), to get that
\begin{equation*}
\varphi^0_{n,i,j}(t,x)=\mathcal{A}^0_{n,i,j}(t,x)+\mathcal{R}_{n,i,j}(t,x),
\end{equation*}
where, for any $p>1$,
\begin{equation*} 
\E [  \vert \mathcal{R}_{n,i,j}(t,x) \vert^{p}]\leq c_{p,T}(v_n^{p/2} + 2^{-np}).
\end{equation*}
We now write
\begin{equation*} \begin{split}
\mathcal{A}^0_{n,i,j}(t,x)  = \sigma_{ij}(u(t,x))+\mathcal{O}_{n,i,j}(t,x),
\end{split}
\end{equation*}
where
\begin{equation*}
\mathcal{O}_{n,i,j}(t,x)= v_n^{-1} \int_{0}^{2^{-n}} \bigg\langle
(\sigma_{ij}(u(t-r,\ast))-\sigma_{ij}(u(t,x))) \Gamma(r, x-\ast),
\Gamma(r,x-\ast) \bigg\rangle_{\mathcal{H}} dr.
\end{equation*}
By Lemma \ref{mes} and hypothesis {\bf (H6)}, it holds that for
all $p>1$,
\begin{equation*} 
\E  [\vert \mathcal{O}_{n,i,j}(t,x) \vert ^p ] \leq
c_{p,T} 2^{-n(\alpha-\eta) p} \; \; \; (\alpha>\eta).
\end{equation*}
Now, as $y \in \textnormal{Supp}(\P_{u(t,x)}) \cap \Sigma$, there exists $r_0>0$
such that for all $0<r \leq r_0$,
\begin{equation*}
B(y;r) \subset \Sigma, \; \text{ and } \; \P \, \{u(t,x) \in B(y;r)\}>0.
\end{equation*}
Moreover, $\sigma_{1}(y),...,\sigma_{k}(y)$ span $\R^k$. Let $\sigma(y)$ denotes the matrix with columns
$\sigma_{1}(y),...,\sigma_{k}(y)$. Hence, for all $0<r \leq r_0$,
\begin{equation*} 
\P \, \biggl\{ (\Vert u(t,x)-y \Vert \leq r) \cap ( \textnormal{det} \, \sigma(u(t,x)) \geq 2 c_1)\biggr\} >0,
\end{equation*}
where
$$
c_1:=\frac{1}{2} \left(\inf_{z \in B(y;r)} \inf_{\Vert \xi \Vert=1} \Vert \sigma(z) \xi \Vert^2\right)^k.
$$
Thus, we conclude that
$$\limsup_{n \rightarrow \infty} \P \, \biggl\{ (\Vert u(t,x)-y \Vert \leq r)
\cap ( \textnormal{det} \, \varphi^0_n(t,x) \geq c_1)\biggr\} >0,$$
which proves (i).

\vskip 12pt
\noindent {\it Proof of} (ii). We start proving that there exist $c>0$ and $\delta>0$, such that
\begin{equation} \label{equa13}
\lim_{n \rightarrow \infty} \P \, \biggl\{ \sup_{\Vert z \Vert \leq \delta} \Vert \varphi^z_n(t,x) \Vert  \leq c\biggr\} =1.
\end{equation}

Observe that (\ref{ec1}), together with (\ref{ec2}), (\ref{ec3}), (\ref{ec4}) and (\ref{ec5}), show that
\begin{equation} \label{estilastb}
 \Vert \varphi^z_n(t,x) \Vert  \leq K+\sup_{\Vert z \Vert \leq \delta}\Vert \mathcal{G}^z_{n}(t,x) \Vert,
\end{equation}
where for any $p>1$,
\begin{equation*} 
\sup_{\Vert z \Vert \leq \delta}\E [  \Vert \mathcal{G}^z_{n}(t,x) \Vert^{p}]\leq c_{p,T}(v_n^{p/2} + \delta^p+ 2^{-np}) .
\end{equation*}

Hence, in order to prove (\ref{equa13}) one only needs to check the
uniformity in $z$. Let $z$ and $z'$ such that $\Vert z \Vert \vee
\Vert z' \Vert \leq \delta$. Then, using (\ref{ec1}) and similar computations as in (\ref{ec3}), (\ref{ec4}) and (\ref{ec5}), and appealing to the Lipschitz property of the derivatives of the coefficients of $\sigma$ and $b$ and (\ref{ec8}),
together with the Cauchy-Schwarz inequality, and finally choosing $n$ large and $\delta$ small, we obtain that
\begin{equation*} \begin{split}
&\sup_{(s,y) \in [0,T] \times \mathbb{R}^d} \E \, [\Vert \varphi^z_n(s,y) -\varphi^{z'}_n(s,y)\Vert ^p] \leq c_{p} \sup_{(s,y) \in [0,T] \times \mathbb{R}^d} \E \, [\Vert \mathcal{A}^z_n(s,y) -\mathcal{A}^{z'}_n(s,y)\Vert ^p] \\
&\qquad \qquad \qquad \qquad \qquad +C_{p,T} \sup_{(s,y) \in [0,T] \times \mathbb{R}^d} \E \, \biggl[\Vert \hat{u}^z_n(s,y)-\hat{u}^{z'}_n(s,y)\Vert ^{2p} \biggr]^{1/2}.
\end{split}
\end{equation*}

We now claim that for all $p>1$,
\begin{equation} \label{cl}
\sup_{(s,y) \in [0,T] \times \mathbb{R}^d} \E \, [\Vert \hat{u}^z_n(s,y)-\hat{u}^{z'}_n(s,y) \Vert^{p}] \leq c_{p,T} \Vert z-z' \Vert^{p}.
\end{equation}
Indeed, using the Lipschitz property of the coefficients of $\sigma$ and $b$, together with \cite[(3.9), (5.15)]{Nualart:07},
we obtain that
\begin{equation*} \begin{split}
&\E \, [\Vert \hat{u}^z_n(t,x)-\hat{u}^{z'}_n(t,x) \Vert^{p}] \leq c_p \Vert z-z' \Vert^p \\
&\qquad \qquad +c_{p,T} \int_0^t \sup_{y \in \R^d}
\E \, [\Vert \hat{u}^z_n(s,y)-\hat{u}^{z'}_n(s,y) \Vert^{p}] \int_{\R^d} \vert \mathcal{F} \Gamma(t-s)(\xi)\vert^2 \mu(d\xi) ds \\
&\qquad \qquad +c_{p,T} \int_0^t \sup_{y \in \R^d} \E \, [\Vert
\hat{u}^z_n(t-s,y)-\hat{u}^{z'}_n(t-s,y) \Vert^{p}] \int_{\R^d}
\Gamma(s, dy) ds.
                  \end{split}
\end{equation*}
Thus, hypotheses (\ref{hyp0}) and (\ref{hyp1}), and Gronwall's lemma prove (\ref{cl}).

We next use the Lipschitz property of $\sigma$ together with (\ref{cl}), to obtain that
\begin{equation*}
\sup_{(s,y) \in [0,T] \times \mathbb{R}^d} \E \, [\Vert
\mathcal{A}^z_n(s,y) -\mathcal{A}^{z'}_n(s,y) \Vert ^{p}] \leq
c_{p,T} \Vert z-z' \Vert^{p}.
\end{equation*}
Therefore, we have proved that
\begin{equation*}
\E \, [\Vert \varphi^z_n(t,x) -\varphi^{z'}_n(t,x)\Vert ^p] \leq c_{p,T} \Vert z-z' \Vert^{p},
\end{equation*}
which, together with (\ref{estilastb}) concludes the proof of (\ref{equa13}).

The proof of (ii) for $\psi^z_n(t,x) $ follows along the same lines, therefore
we only give the main steps.
Let
$$
\psi^z_{n,i,j,m}(t,x)=\frac{\partial^2}{\partial z_m \partial z_j} \hat{u}^z_{n,i}(t,x).
$$
Then we have that
\begin{equation*}
\psi^z_{n,i,j,m}(t,x)=\int_0^t \int_0^t \langle D_s^{(m)} D^{(j)}_r(\hat{u}^z_{n,i}(t,x)) , g^j_{n}(r, \ast) \otimes g^m_{n}(s, \ast) \rangle_{\mathcal{H}\otimes \mathcal{H}} dr ds,
\end{equation*}
where the $\mathcal{H}\otimes \mathcal{H}$-valued process $D_s^{(m)} D^{(j)}_r(\hat{u}^z_{n,i}(t,x))$, satisfies the following linear stochastic differential equation
\begin{equation*}
 \begin{split}
D_s^{(m)} D^{(j)}_r(\hat{u}^z_{n,i}(t,x)) &= \Gamma(t-r, x-\ast) D_s^{(m)}(\sigma_{ij}(u(r, \ast)))  +
\Gamma (t-s, x-\ast ) D^{(j)}_r (\sigma_{i m}(u(s,\ast)))\\
&+\int_{r \vee s}^t \int_{\R^d} \Gamma (t-w, x-y) \sum_{\ell=1}^{k} D_s^{(m)} D^{(j)}_r (\sigma_{i \ell}(u(w,y))) W^{\ell}(dw, dy) \\
&+\int_{r \vee s}^t \int_{\R^d} \Gamma (t-w, dy) D_s^{(m)} D^{(j)}_r
(b_i(u(w,x-y))) \,dw.
 \end{split}
\end{equation*}
Using the chain rule and the stochastic differential equation satisfied by the first derivative,
one can compute the different terms of $\psi^z_{n,i,j,m}(t,x)$ as we did for $\varphi^z_{n,i,j}(t,x)$, and bound their $p$th-moments.
Finally, one estimates the $p$th-moments of the difference  $\psi^z_n(t,x)- \psi^{z'}_n(t,x)$
as we did for $\varphi^z_n(t,x)$ in order to get the desired result.

\subsection{Proof of Theorem \ref{tdim1}}

The existence and smoothness of the density follows from
\cite[Theorem 6.2]{Nualart:07}. Hence, we only need to prove the
strict positivity. For this, one applies Theorem \ref{tm1} as in
Theorem \ref{maintt} taking $\Sigma=\R$. In this case, in order
to prove hypothesis (i), one proceeds as in
Theorem \ref{maintt}, using hypotheses {\bf (H1)},
to show that for all $p>1$,
\begin{equation*}
\lim_{n\rightarrow \infty} \E  \biggl[\vert \varphi^0_{n}(t,x)- \mathcal{A}^0_{n}(t,x) \vert^p \biggr] =0.
\end{equation*}
Next using the non-degeneracy assumption on
$\sigma$, one gets that
$\mathcal{A}^0_{n}(t,x) \geq c$, which implies that
\begin{equation*}
\varphi^0_{n}(t,x) \geq c -\vert \varphi^0_{n}(t,x)-
\mathcal{A}^0_{n}(t,x) \vert.
\end{equation*}
Finally, these assertions imply that for all $y \in \textnormal{Supp}(\P_{u(t,x)})$
$$\limsup_{n \rightarrow \infty} \P \, \biggl\{ (\vert u(t,x)-y \vert \leq r)
\cap ( \varphi^0_n(t,x) \geq \frac{c}{2})\biggr\} >0,$$ which
proves (i). The proof of (ii) follows exactly as in Theorem
\ref{maintt}.
\vskip 12pt
\noindent {\bf Acknowledgement}.
The author would like to thank Professors Vlad Bally, Robert C. Dalang, Marta Sanz-Sol\'e and Llu\'is Quer-Sardanyons for stimulating discussions on the subject, and the anonymous referees that handled this article by all their comments and corrections that have helped to improve the paper.


\begin{thebibliography}{AAA99}
\bibitem{Bally:98} {\sc{Bally, V. and Pardoux, E.}} (1998),
Malliavin calculus for white noise driven parabolic {S}{P}{D}{E}s,
{\it Potential Analysis}, {{\bf 9}}, 27-64.


\bibitem{Cardon:02} {\sc{Cardon-Weber, C.}} (2002),
Cahn-Hilliard stochastic equation: strict positivity of the density,
{\it  Stoch. Stoch. Rep.}, {{\bf 72}}, 191--227.


\bibitem{Carmona:88} {\sc{Carmona, R. and Nualart, D.}} (1988),
Random nonlinear wave equations: smoothness of the solutions,
{\it  Probab. Theory Related Fields }, {{\bf 79}}, 469--508.

\bibitem{Chaleyat:03} {\sc{Chaleyat-Maurel, M. and Sanz-Sol\'e, M.}} (2003),
Positivity of the density for the stochastic wave equation in two spatial dimensions,
{\it ESAIM: Probability and Statistics}, {{\bf 7}}, 89-114.

\bibitem{Dalang:99} {\sc{Dalang, R.C.}} (1999), Extending martingale measure stochastic integral with applications to spatially homogeneous s.p.d.e's, {\it Electronic Journal of Probability}, {\bf{4}}, 1-29.

\bibitem{Dalang:04} {\sc{Dalang, R.C. and Nualart, E.}} (2004), Potential theory for hyperbolic SPDEs,
{\it The Annals of Probability}, {\bf{32}}, 2099-2148.

\bibitem{Dalang:08} {\sc{Dalang, R.C., Khoshnevisan, D., and Nualart, E.}} (2008), Hitting probabilities for systems of non-linear
stochastic heat equations with additive noise,
\textit{ALEA}, {\bf{32}}, 2099-2148.

\bibitem{Dalang:09} {\sc{Dalang, R.C., Khoshnevisan, D., and Nualart, E.}} (2009), Hitting probabilities for systems of non-linear
stochastic heat equation with multiplicative noise, \textit{Probability theory and related fields}, {\bf{32}}, 2099-2148.

\bibitem{Dalang:09b} {\sc{Dalang, R.C., Khoshnevisan, D., and Nualart, E.}} (2010), Hitting probabilities for systems of non-linear spatially homogeneous
stochastic heat equations, \textit{Preprint}.

\bibitem{Dalang:07} {\sc{Dalang, R.C. and Sanz-Sol\'e, M.}} (2009), {\it H\"older-Sobolev regularity of the solution
to the stochastic wave equation in dimension 3}, Memoirs of the AMS, {\bf{199}}.

\bibitem{Dalang:10} {\sc{Dalang, R.C. and Quer-Sardanyons, L.}} (2010), Stochastic integrals for s.p.d.e's: a comparison, 
{\it To appear in Expositiones Mathematicae}.

\bibitem{Daprato:92}
{\sc{Da Prato, G. and Zabczyk, J.}} (1992), {\it Stochastic equations in infinite dimensions}, Encyclopedia of Mathematics and its Applications, 44. Cambridge University Press, Cambridge.

\bibitem{Fournier:99} {\sc{Fournier, N.}} (1999), Strict positivity of the density for a Poisson
driven S.D.E.,
{\it Stochastic and Stochastic Reports}, {\bf{68}}, 1-43.

\bibitem{Marquez:01} {\sc{M\'arquez-Carreras, D., Mellouk, M. and Sarr\`a, M.}} (2001), On stochastic partial differential equations
with spatially correlated noise: smoothness of the law,
{\it Stochastic Processes and their Applications}, {\bf{93}}, 269-284.

\bibitem{Millet:97} {\sc{Millet, A. and Sanz-Sol\'e, M.}} (1997), Points of positive density for the solution to a hyperbolic SPDEs, {\it Potential Analysis}, {\bf{7}}, 623-659.

\bibitem{Millet:99} {\sc{Millet, A. and Sanz-Sol\'e, M.}} (1999), A stochastic wave equation in two
space dimension: smoothness of the law,
{\it The Annals of Probability}, {\bf{27}}, 803-844.

\bibitem{Nualart:98} {\sc{Nualart, D.}} (1998), Analysis on Wiener space and
anticipating stochastic calculus, {\it Ecole kEt\'e de
Probabilit\'es de Saint-Flour XXV, Lect. Notes in Math.}
{\bf{1690}}, Springer-Verlag, 123-227.

\bibitem{Nualart:06}
{\sc{Nualart, D.}} (2006), {\it The Malliavin calculus and related
topics, Second Edition}, Springer-Verlag.

\bibitem{Nualart:07}
{\sc{Nualart, D. and Quer-Sardanyons, L.}} (2007), Existence and smoothness of the density for spatially homogeneous SPDEs,
\textit{Potential Analysis}, {\bf 27}, 281-299.

\bibitem{Nualart:09}
{\sc{Nualart, D. and Quer-Sardanyons, L.}} (2009), Gaussian density estimates for solutions of quasi-linear stochastic partial
differential equations,
\textit{Stochastic Processes and Applications}, {\bf{119}}, 3914-3938.

\bibitem{Nualart:10}
{\sc{Nualart, D. and Quer-Sardanyons, L.}} (2010), Optimal Gaussian density estimates for a class of stochastic equations with additive noise,
\textit{To appear in Infinite Dimensional Analysis, Quantum Probability and Related Topics}.

\bibitem{ENualart:10}
{\sc{Nualart, E. and Quer-Sardanyons, L.}} (2010), Gaussian estimates for the density of the non-linear stochastic heat equation in any space dimension,
\textit{Preprint}.


\bibitem{Sanz:00} {\sc{Sanz-Sol\'e, M. and Sarr\`a, M.}} (2000), Path Properties of a Class of Gaussian Processes with Applications
to SPDE's, {\it Canadian mathematical Society Conference Proceedings}, {\bf{28}}, 303-316.


\bibitem{Sanz:02} {\sc{Sanz-Sol\'e, M. and Sarr\`a, M.}} (2002), H\"older continuity for the stochastic
heat equation with spatially correlated noise,
{\it Seminar on Stochastic Analysis, Random Fields ans Applications, III (Ascona, 1999), Progr. Prob.}, {\bf{52}}, 259-268.


\bibitem{Quer:04} {\sc{Quer-Sardanyons, L. and Sanz-Sol\'e, M.}} (2004),
Absolute continuity of the law of the solution to the $3$-dimensional stochastic wave equation,
{\it Journal of Functional Analysis}, {{\bf 206}}, 1-32.

\bibitem{Sanz:04} {\sc{Quer-Sardanyons, L. and Sanz-Sol\'e, M.}} (2004),
A stochastic wave equation in dimension 3: Smoothness of the law,
{\it Bernouilli}, {{\bf 10}}, 165-186.

\bibitem{Walsh:86} {\sc{Walsh, J.B.}} (1986), An Introduction to Stochastic Partial Differential
Equations, {\it Ecole d'\'Et\'e de Probabilit\'es de Saint-Flour
XIV, Lect. Notes in Math.}, {\bf{1180}}, Springer-Verlag, 266-437.

\end{thebibliography}
\end{document}